\documentclass[12pt]{amsart}

\usepackage{amsmath}
\usepackage{amssymb}
\usepackage{amstext}
\usepackage{amscd}
\usepackage[frenchb,english]{babel}
\usepackage{enumitem}
\usepackage{greekctr}
\usepackage[matrix,arrow,ps]{xy}
\usepackage[bookmarks=true,backref,colorlinks=true,citecolor=blue,urlcolor=blue,linkcolor=magenta]{hyperref}

\newtheorem{theorem}{Theorem}[section]
\newtheorem{proposition}[theorem]{Proposition}
\newtheorem{lemma}[theorem]{Lemma}
\newtheorem{corollary}[theorem]{Corollary}
\newtheorem{conjecture}[theorem]{Conjecture}
\newtheorem{definition}[theorem]{Definition}

\newcommand{\Hom}{\mbox{Hom}}

\newcommand{\End}{{\rm End}}

\newcommand{\Aut}{{\rm Aut}}

\newcommand{\Disc}{{\rm Disc}}

\newcommand{\der}{{\rm der}}

\newcommand{\F}{{\mathbb F}}
\newcommand{\CC}{{\mathbb C}}

\newcommand{\R}{{\mathbb R}}

\newcommand{\Z}{{\mathbb Z}}

\newcommand{\Q}{{\mathbb Q}}
\newcommand{\C}{{\mathbb C}}

\newcommand{\SSS}{{\mathbb S}}
\newcommand{\AAA}{{\mathbb A}}
\newcommand{\A}{{\mathbb A}}

\newcommand{\ol}{\overline}
\newcommand{\wh}{\widehat}

\newcommand{\wt}{\widetilde}

\newcommand{\tens}{\otimes}
\usepackage{mathtools}

\DeclarePairedDelimiterX{\Nm}[1]{\lVert}{\rVert}{#1}
\newcommand{\sous}{\backslash}

\newcommand{\abs}[1]{\left|#1\right|}

\title[Heights on `hybrid orbits']{Heights on `hybrid orbits' in Shimura varieties.}
\author{Rodolphe Richard, Andrei Yafaev}
\date{\today}
 \address{Richard, Yafaev: UCL, Department of Mathematics, Gower Street, WC1E 6BT, London, UK}
\email{Rodolphe.Richard@normalesup.org, yafaev@math.ucl.ac.uk}

\begin{document}
\setcounter{tocdepth}{1}
\setcounter{secnumdepth}{4}
\maketitle

\begin{abstract}
{
We prove the ``hybrid conjecture'' which is a common generalisation of the André-Oort conjecture and the André-Pink-Zannier conjecture, in the case of Shimura varieties
of abelian type.
}
\end{abstract}

\tableofcontents

\section{Introduction}\label{sec:intro}

In the context of abelian varieties, the Manin-Mumford conjecture was proved in~\cite{Raynaud}. The Mordell conjecture was proven in~\cite{FaFa} and the weak Mordell-Lang conjecture for finite type subgroup in~\cite{FaFa2}. The full Mordell-Lang conjecture, about finite \emph{rank} subgroups, contains both the Manin-Mumford conjecture and the weak Mordell-Lang conjecture. See~\cite{Poonen} for another hybrid result, which combines Bogomolov conjecture and the full Mordell-Lang conjecture. The full Mordell-Lang conjecture for a subvariety~$V$ of an abelian variety~$A$ and a finite rank subgroup generated by~$x_1,\ldots,x_n\in A$ is a consequence of the Zilber-Pink conjecture for the subvariety~$V\times\{(x_1,\ldots,x_n)\}$ in the abelian variety~$A\times A^n$.

In the context of Shimura varieties, the André-Oort conjecture was proved for Shimura varieties of abelian type in the mid-2010s and in general in the early 2020s, as a combination of many works\footnote{See the introduction of~\cite{AO-PST}.}. The André-Pink-Zannier conjecture was proven for Shimura varieties of abelian type in \cite{APZ1}, \cite{APZ2}. The André-Oort conjecture is an analogue of the Manin-Mumford conjecture and the André-Pink-Zannier conjecture an analogue of the weak Mordell-Lang conjecture. The corresponding analogue of the full Mordell-Lang conjecture is the \emph{hybrid conjecture} introduced in~\cite{RYLMS}. See also~\cite{BD} for a slightly different viewpoint.    
The  hybrid conjecture for a subvariety~$V$ of a Shimura variety~$Sh_K(G,X)$ and the hybrid orbit of a pointed Shimura Datum~$(M,X_M,x)$ is a consequence of the Zilber-Pink conjecture for the subvariety~$V\times\{[x,1]\}$ of the weakly special subvariety~$Sh_K(G,X)\times \{[x,1]\}\subseteq Sh_K(G,X)\times Sh_{K_M}(M,X_M)$.
 
The hybrid conjecture is a substantial strengthening of both André-Oort and the generalised André-Pink-Zannier conjecture. It is stronger than the conjunction of these. It is significantly stronger than the main result of~\cite{AD} (see~\cite[\S8]{RYLMS}). The hybrid conjecture implies the  Zilber-Pink conjecture for hypersurfaces of weakly special subvarieties. (See~\cite{RYLMS}). See~\cite[Th. 7.8--7.9]{BD} for other consequences of the hybrid conjecture in relation to the Zilber-Pink conjectures. 

We refer to~\cite[Part IV]{P} for generalities on the Zilber-Pink conjectures.

In this article we prove the hybrid conjecture for Shimura varieties of abelian type. This applies to the most important Shimura varieties, notably those related to moduli spaces of abelian varieties. 
As a result, one deduces the consequences mentioned above in the abelian type case.

To provide a motivation for the definition below, consider the following situation.
Let~$A$ be an abelian variety over~$\C$. The sets of the form~$\{\phi(x)|\phi\in \Hom(B,A)\}\subseteq A$, where~$B$ is an abelian variety and~$x\in B$, are the subgroups of finite type which are invariant under endomorphisms of~$A$. Every finite type subgroup of an abelian variety is contained in a finitely generated~$\End(A)$-module, and finitely generated~$\End(A)$-modules are of finite type as abelian groups. The weak Mordell-Lang conjecture is concerned with finitely generated subgroups in abelian varieties. It is equivalent to consider subgroups of the form~$\{\phi(x)|\phi\in \Hom(B,A)\}$ for some abelian variety~$B$ and some fixed~$x\in B$. We obtain a finite rank subgroup by taking inverse images of the~$\phi(x)$ under isogenies~$A\to A$.

\begin{definition}Let~$(M,X_M)$ and~$(G,X)$ be two Shimura data and~$x\in X_M$ be a Hodge generic point. The hybrid orbit of~$(M,X_M,x)$ in~$(G,X)$ is the set of~$x'\in X$ such that,
denoting by~$(M_{x'},X_{x'})\leq (G,X)$ the smallest Shimura subdatum such that~$x'\in X_{x'}$, there exists a morphism of Shimura data
\[
\Phi:(M,X)\to (M^{ad}_{x'},X^{ad}_{x'})
\]
such that~$\Phi\circ x=x'^{ad}$.
\end{definition}

Recall that~$x'$ is in the generalised Hecke orbit of~$x$ when there is a morphism of Shimura data~$\Phi:(M,X)\to (M_{x'},X_{x'})$ such that~$\Phi\circ x=x'$. Therefore, the generalised Hecke orbit of~$x$ is a subset of the hybrid orbit of~$x$. Observe that~$x'$ is a special point if and only if~$(M^{ad}_{x'},X^{ad}_{x'})=(1,\{1\})$. Therefore, a hybrid orbit always contains the set of all special points as a subset. As a consequence, the following is a common generalisation of the André-Oort Conjecture and the generalised André-Pink-Zannier Conjecture. We refer to~\cite{RYLMS} for other equivalent formulations of Conjecture~\ref{conj:hybrid}.
\begin{conjecture}[Hybrid conjecture]\label{conj:hybrid}
 Let~$V\subseteq Sh_K(G,X)$ be an irreducible subvariety and let~$\Sigma\subseteq Sh_K(G,X)$ the image by~$X\to Sh_K(G,X)$ of a hybrid orbit. Assume that
\[
V\cap \Sigma
\]
is Zariski dense in~$V$. Then~$V$ is weakly special.
\end{conjecture}
By~\cite{RYLMS} this implies the following. We refer to~\cite{DR} for the notion of optimal subvariety.
\begin{conjecture}[Zilber-Pink conjecture for weak defect one] \label{conj:case ZP}
Let~$W\subseteq Sh_K(G,X)$ be a Hodge generic weakly special subvariety and~$V\subseteq W$ an irreducible subvariety of codimension one.
Then~$V$ has finitely many optimal subvarieties.
\end{conjecture}

Our main result is the following.
\begin{theorem}\label{thm:main thm}
Assume that~$(G, X)$ is of abelian type. Then Conjecture~\ref{conj:hybrid} holds true.
\end{theorem}
We obtain the following as a consequence.
\begin{corollary} Assume that~$(G,X)$ is of abelian type. Then Conjecture~\ref{conj:case ZP} holds true.
\end{corollary}

\subsection*{Summary of the article} In~\S\ref{sec:param} we introduce a parametrisation of special points by rational points of some homogeneous varieties.
In~\S\ref{sec:geometric hybrid orbit} we introduce
the geometric hybrid orbits and their parametrisation by rational points of some associated homogeneous varieties. We prove that every hybrid orbit is a finite union of  geometric hybrid orbits.
In~\S\ref{sec:height functions} we study points of a geometric hybrid orbit for which the preimage under the parametrisation contains a unique rational point. It allows us to introduce natural
height functions on geometric hybrid orbits. The height can be decomposed as a product of a toric central part~$H_f^{cent}$ and a semisimple derived part~$H_f^{der}$. The central part~$H_f^{cent}$ is the height of the centre of the Mumford-Tate group, viewed as a torus embedded in~$GL(N)$.
In~\S\ref{sec:height disc} we prove the height of a torus~$T\leq GL(N)$ is polynomially equivalent to the discriminant used in the proof of the André-Oort conjecture. This discriminant is a product of the absolute value~$d(T):=\abs{disc(L(T))}$ of the discriminant of the splitting field and an index~$[T_{max}:T(\widehat{\Z})]$ of a open subgroup~$T(\widehat{\Z})$ of the maximal compact subgroup~$T_{max}\leq T(\A_f)$.
Section~\ref{sec:Galois conj} is central to our article. We state Conjecture~\ref{conj:Bounds} on lower bounds of the size of Galois orbits in a hybrid orbit, in terms of our natural height functions. We explain how the formalism of polynomial equivalence and height 
functions let us decompose the problem in bounds involving~$H_f^{cent}$ and bounds involving the derived part~$H_f^{der}$ up to a power of the central part.
In~\S\ref{sec:Tate} we use results of~\cite[App. C]{APZ2} to prove decomposition properties of images of representation of images the Galois representations associated with a hybrid orbit. This section requires information concerning the Tate conjecture for this Galois representation. 
In Shimura varieties of abelian type, this was derived in~\cite{APZ2} from Faltings theorems on Tate isogeny conjecture.
In~\S\ref{sec:adelic tori} we combine this with the study of adelic orbits from~\cite[App. B]{APZ1}. We obtain a lower bound for the size of Galois orbits in terms of the index~$[T_{max}:T(\widehat{\Z})]$. In~\S\ref{sec:Auxiliary torus} we obtain a lower bound in terms of the absolute value of the discriminant of the splitting field. The argument is a reduction to bounds obtained for special points in the proof of André-Oort conjecture for abelian type Shimura varieties\footnote{Consequences, given in~\cite{TsiAG}, of the averaged Colmez conjecture.}. 
In~\S\ref{sec:geometric bounds} we obtain a lower bound for the size of Galois orbits~$H_f^{der}$, up to a power of~$H_f^{cent}$. This uses results from~\S\ref{sec:Tate} and the methods of~\cite{APZ1} and~\cite{APZ2}.
In~\S\S\ref{sec:archimedan part}--\ref{sec:invariance} we adapt to the hybrid setting some technical results needed in the proof~\S\ref{sec:proof} of the André-Pink-Zannier conjecture. Finally \S\ref{sec:proof} proves the hybrid conjecture along the lines of~\cite{APZ1} and~\cite{APZ2}
for the André-Pink-Zannier conjecture. It implements the Pila-Zannier strategy and uses~\S\S\ref{sec:archimedan part}--\ref{sec:invariance}. An important step uses Conjecture~\S\ref{conj:Bounds}, which, by sections~\S\S\ref{sec:Galois conj}--\ref{sec:geometric bounds}, holds true in Shimura varieties of abelian type. 

\subsection*{Acknowledgements} The second named author is supported by the University of Manchester and Grant~EP/Y020758/1, and was supported by IHÉS.

\section{Parametrisation of special points}\label{sec:param}
\emph{We study the conjugacy classes over~$\ol{\Q}$ of CM-Tori. This allows us to parametrise special points by rational points of finitely many homogeneous varieties. This will be extended to hybrid orbits in the next section.}

\begin{theorem}\label{thm:finiteness conjugacy special}
Let~$(G, X)$ be a Shimura datum, and let~$\Sigma\subseteq X$ be the subset of all special points.
Then there exist finitely many algebraic tori~$T_1,\ldots,T_k$ such that, for every~$\tau \in \Sigma$, there exists an~$i\in \{1;\ldots;k\}$ and q~$g\in G(\ol{\Q})$ such that~$gM_\tau g^{-1}=T_i$.
\end{theorem}
\begin{proof}Let~$\tau$ be a special point and~$T=M_\tau\leq G$ be its Mumford-Tate group. Then~$T$ is a Torus defined over~$\Q$, and~$\tau$ corresponds to a morphism~$\tau:\SSS\to T_\R\leq G_\R$, where~$\SSS$ denotes the Deligne torus. 
As~$S:=\tau(\SSS)$ is~$\Q$-Zariski dense in~$T$, the torus~$T_\CC$ is generated by~$\bigcup_{\sigma\in Gal(\ol{\Q}/\Q)} \sigma(S_\CC)$. 

There exists a maximal torus~$T_0$ defined over~${\Q}$ (\cite[Th. 18.2]{BorelLAG}). Let~$\rho:Gal(\ol{\Q}/\Q)\to \Aut(T_0)$ be the natural action. We recall that~$V:=\rho(Gal(\ol{\Q}/\Q))$ is a finite group.  Let~$W\simeq N_G(T_0)/Z_G(T_0)$ be the Weyl group of~$T_0$ in~$G$. It is a finite group and it acts on~$T_0$ by automorphisms. We recall that maximal tori defined over~$\ol{\Q}$ are~$G(\ol{\Q})$-conjugated, and that two subtori~$R,R'\leq T_0$ which are conjugated in~$G(\ol{\Q})$ are conjugated by an element of~$N_{G(\ol{\Q})}(T_0)$. 

Therefore there exists a subtorus~$S_0\leq {T_0}_{\ol{\Q}}$ which is $G(\ol{\Q})$-conjugated to~$S$. Then~$S_0$ is well defined up to the action of~$W$ and depends only on the conjugacy class of~$S$. In particular, it depends only on~$X$, but not on~$\tau$.

There are thus~$\ol{\Q}$-subtori~$S'\leq T'\leq {T_0}_{\ol{\Q}}$ and~$g\in G(\ol{\Q})$ such that~$S=gS'g^{-1}$ and~$T=gT'g^{-1}$.
We may assume~$S'=S_0$. For~$\sigma\in Gal(\ol{\Q}/\Q)$, the torus
\[
g\sigma(S)g^{-1}\leq g\sigma(T)g^{-1}=gTg^{-1}=T'\leq T_0
\] 
is conjugated to~$\sigma(S')=\sigma(g)\sigma(S)\sigma(g)^{-1}\leq \sigma(T_0)=T_0$. There is thus~$w_\sigma\in W$ such that~$g\sigma(S)g^{-1}=w_\sigma(\sigma(S'))=w_\sigma\circ v_\sigma(S')$, where~$v_\sigma=\rho(\sigma)$.

Then~$T'$ is generated by the~$w_\sigma\circ v_\sigma(S')$. Since~$w_\sigma$ and~$v_\sigma$ can take only finitely many values, and there are only finitely many possibilities~$T_1,\ldots, T_k$ for~$T'$. As~$T$ is conjugated to~$T'$, the conclusion follows.
\end{proof}
\subsubsection*{Remarks on the above proof}\label{rem:Chai Oort} Here are a few more details that can be found in~\cite{CO}.
\begin{enumerate}
\item \label{rem:Chai Oort1}
The Galois group acts on~$T_0$ via the Weyl group: we have~$V\leq W$.
\item \label{rem:Chai Oort2}
There are maximal tori~$T_0$ for which~$V=W$. These are called \emph{Weyl tori}. These are the most common maximal tori: in~$(T_0/N_G(T_0))(\Q)$, the set of exceptions is a ``thin'' subset in the sense of~\cite[Prop. 9.2]{SerreLMW}.
\item \label{rem:Chai Oort3}
There are special points~$\tau$ such that~$M_{\tau}$ is a Weyl torus, by~\cite[case~$m=0$ of Theorem 5.5]{CO}. 
\item \label{rem:Chai Oort4}
 In particular, there are special points~$\tau$ such that~$M_{\tau}$ is a maximal torus.
\end{enumerate}

\begin{proposition}\label{prop:param special}
For~$i\in\{1;\ldots;k\}$, let~$\tau_i$ be such that~$M_i:=M_{\tau_i}$ is conjugated to~$T_i$. 
We denote by~$N_i:=N_G(M_i)$ and let~$W_i\simeq G/N_G(M_i)$ be the conjugacy class of~$M_i$, as an algebraic variety over~$\Q$. We define~$W_i(\R)^+:=G(\R)^+/N_i(\R)\subseteq W(\R)$ the~$G(\R)^+$-conjugacy class of~$M_i$ and~$W_i(\Q)^+:=W_i(\Q)\cap W_i(\R)^+$. 
Then the maps
\begin{equation}\label{param maps pi}
p_i:gM_ig^{-1}\mapsto g\cdot \tau_i : W_i^+\to X
\end{equation}
are well-defined and we have
\begin{equation}\label{eq truc sigma}
\Sigma=\bigcup_{i=1}^{k} p_i(W_i(\Q)^+).
\end{equation}
\end{proposition}
\begin{proof}[Proof that the maps~$p_i$ are well defined] We may assume that~$G=G^{ad}$. Let~$I\in \SSS(\R)\simeq \CC^\times$ be such that~$I^2=-1$. Then~$\theta:g\mapsto \tau_i(I)g\tau_i(I)^{-1}$ is a Cartan involution, and~$K=\{g\in G(\R)|\theta(g)=g\}$ is a maximal compact subgroup of~$G(\R)$ and~$K^+$ is a maximal compact subgroup of~$G(\R)^+$. We recall that~$K^+=Z_{G(\R)}(\tau_i(\SSS))$ and this~$K^+$ is also the stabiliser of~$Stab_{G(\R)}(\tau_i)$ of the point~$\tau_i\in X$. If we denote by~$[M_i]\in W$ the point associated to~$M_i$, we also have
\[
Stab_{G(\R)^+}([M_i])=N_{G(\R)^+}(M_i).
\] 

As~$\tau_i(\SSS)\leq M_i$ and~$M_i$ is a torus, we have~$M_i\leq Z_G(M_i)$ and~$Z_{G(\R)}(M_i)\leq  Z_{G(\R)}(\tau_i(\SSS))=K^+\leq K$. Therefore~$M_i(\R)$ is  fixed by~$\theta$ pointwise. We deduce that~$N_{G(\R)}(M_i)$ is stable under~$\theta$. 
It implies that~$N:=N_{G(\R)^+}(M_i)=N_{G(\R)}(M_i)\cap G(\R)^+$ is stable under~$\theta$. Then the restriction of~$\theta$ to~$N$ is a Cartan involution, and~$K_N:=K\cap N=K^+\cap N_{G(\R)}(M_i)$ is a maximal compact subgroup of~$N$. 
As~$M_i$ is reductive, we have~$N_G(M_i)^0=M_i\cdot Z_G(M_i)^0$. As the group~$M_i(\R)\leq Z_{G(\R)}(M_i)\leq K^+$ are compact, the group~$N_{G(\R)}(M_i)^0$ is compact, and so are~$N_{G(\R)}(M_i)$ and~$N$. Thus~$N=K_N\leq K^+=Z_{G(\R)}(\tau_i(\SSS))$, and the image of~$N$ in~$\Aut(M_i)$ centralises~$\tau_i(\SSS)$. We have proved
\[
N_{G(\R)^+}(M_i)\leq Z_{G(\R)}(\tau_i(\SSS))=Stab_{G(\R)}(\tau_i).
\]
It follows that~$Stab_{G(\R)^+}([M_i])\leq Stab_{G(\R)}(\tau_i)$, and that the map~$p_i$ is well defined.
\end{proof}
\begin{proof}[Proof of~\eqref{eq truc sigma}]
For~$\tau\in \Sigma$ and~$T:=M_{\tau}$, there exists, by Theorem~\ref{thm:finiteness conjugacy special}, some~$g\in G(\ol{\Q})$ and~$i\in\{1;\ldots;k\}$ such that~$T=gM_ig^{-1}$. As~$T$ is a conjugate of~$T_{\tau_i}$, and~$T$ is a torus defined over~$\Q$, we have~$gN_G(T_{\tau_i})\in W_i(\Q)$. This proves~$\Sigma\subseteq\bigcup_{i=1}^{k} p_i(W_i(\Q)^+)$.

We prove the reverse inclusion. Let~$g$ and~$i\in\{1;\ldots;k\}$ be such that~$gN_G(M_i)\in W_i(\Q)$. That is, the torus~$T:=gM_ig^{-1}$ is defined over~$\Q$. Note that~$\tau(\SSS)\leq M_i(\R)$, and thus~$g\tau(\SSS)g^{-1}\leq T$.
As~$T$ is defined over~$\Q$, it contains the Mumford-Tate group~$M_{g\tau}$. In particular, $M_{g\tau}$ is commutative, and thus is a torus, and~$g\tau=p_i(g)$ is a special point.
This concludes the proof of the reverse inclusion and finishes the proof of Proposition~\ref{prop:param special}.
\end{proof}
Using the Remark~\ref{rem:Chai Oort4} following Theorem~\ref{thm:finiteness conjugacy special}, we deduce the following.

\begin{theorem}Let~$W$ be the space of maximal tori of~$G^{ad}$, viewed as an algebraic variety over~$\Q$, and let~$W^+\leq W(\R)$ be the subset of $\R$-anisotropic maximal tori. Then there is a~$G(\R)^+$-equivariant map
\[
p:W^+\to X^+
\]
such that~$p(W(\Q)^+)=\Sigma\cap X^+$ and such that, for~$\tau\in \Sigma$, and~$T\in W(\Q)^+$, we have
\begin{equation}\label{map special max param}
p(T)=\tau\Leftrightarrow M_\tau\leq T.
\end{equation}
\end{theorem}


\section{Geometric hybrid orbits}\label{sec:geometric hybrid orbit}
\emph{The decomposition of generalised Hecke orbits into a disjoint union of finitely many geometric Hecke orbits, and the parametrisation of Hecke orbits were introduced in~\cite{APZ1} was an essential tool in order to define height functions and formulate lower bounds for the size of Galois orbits needed in the proof of the André-Pink-Zannier conjecture. Here we define geometric hybrid orbits and their parametrisation, and prove that every hybrid orbit is a union of finitely many geometric hybrid orbits. This is not necessarily a disjoint union.}

\begin{theorem}\label{thm:geom in hybrid}
Let~$(G,X)$ be a Shimura datum, let~$x'\in X$ and let $(M_{x'},X_{x'},x')\leq (G,X)$ be the smallest Shimura subdatum
such that~$x'\in X_{x'}$. Let~$Z=Z_{G}(M^{der}_{x'})$ be the centraliser
of the derived subgroup~$M^{der}_{x'}$ of~$M_{x'}$ and let~$N=N_{G}(Z(M_{x'}))$ be the normaliser of the centre~$Z(M_{x'})$ of~$M_{x'}$. Let
\begin{equation}\label{eq:W hybride}
W:=G/(N\cap Z),
\end{equation}
as a~$\Q$-algebraic variety, and let~$W^+:=G(\R)^+/(N(\R)\cap Z(\R)\cap G(\R)^+)\subseteq W(\R)$.

Then there exists a unique~$G(\R)^+$ equivariant map~$\pi:W^+\to X$ which maps the neutral coset~$1\cdot (N\cap Z)$ to~$x'$.

Moreover, we have
\begin{equation}\label{eq:geom in hybrid}
\pi(W(\Q)\cap W^+)\subseteq \Sigma_{X}(M_{x'},X_{x'},x').
\end{equation}
\end{theorem}
\begin{proof}[Proof of Theorem~\ref{thm:geom in hybrid}]
Since~$G(\R)^+$ acts transitively on~$W^+$,  the map~$\pi$ is unique and well defined if and only if~\begin{equation}\label{eq:NZG}
N(\R)\cap Z(\R)\cap G(\R)^+\leq Stab_{G(\R)}(x').
\end{equation}
Without loss of generality, we may assume~$G=G^{ad}$.
\begin{proof}[Proof of~\eqref{eq:NZG}]We write~$M=M_{x'}$ for simplicity.
As~$Z(M)$ is a torus,~$Aut(Z(M))$ is of dimension zero, and the image of the map~$(N\cap Z)^0\to Aut(Z(M))$ is connected and of dimension zero. We deduce~$(N\cap Z)^0\leq Z_G(Z(M))$ and 
\begin{equation}\label{eq:NZvsZ}
(N\cap Z)^0\leq Z(Z(M))\cap N\cap Z=Z_G(Z(M))\cdot Z_G(M^{der})=Z_G(M).
\end{equation}
As~$Z_G(M)$ is compact, so is~$(N\cap Z)^0$ and so is~$N\cap Z$.

Let~$\theta:G(\R)\to G(\R)$ be the Cartan involution associated to~$x'$. Then~$M$ is invariant
under~$\theta$. Thus~$Z(M)$ and~$M^{der}$ and~$N_G(M^{der})$ and~$Z_G(M^{der})$ 
are invariant under~$\theta$. Thus,~$N\cap Z$ is a compact group which is stable under~$\theta$.

Let~$K$ be the maximal compact subgroup of~$G(\R)$ fixed by~$\theta$. Then~$N\cap Z\leq K$.
Moreover,~$Stab_{G(\R)^+}(x')=G(\R)^+\cap K=K^+$. This implies~\eqref{eq:NZG}. It follows that the map
\end{proof}

We can interpret~$W(\Q)$ as follows. As the map
\[
G/(N\cap Z)\to G/N\times G/Z
\]
is injective, we have,
\begin{multline}
\forall g\in G(\C), g(N\cap Z)\in (G/(N\cap Z))(\Q)\\ \Leftrightarrow 
 gN\in (G/N)(\Q)\text{ and }
 gZ\in (G/Z)(\Q).
\end{multline}
We have~$gN\in (G/N)(\Q)$ if and only if the~$\C$-torus $gZ(M_{x'})g^{-1}$ is defined over~$\Q$.
We have~$gZ\in (G/Z)(\Q)$ if and only if $gM^{\der}_{x'}g^{-1}$ is defined over~$\Q$ \emph{and} the
restriction of~$\C$-morphism~$AD_g:h\mapsto ghg^{-1}$ induces a isomorphism \emph{defined over~$\Q$}
from~$M^{der}_{x'}$ to~$gM^{\der}_{x'}g^{-1}$.

\begin{proof}[Proof of~\eqref{eq:geom in hybrid}]
Consider~$g\in G(\R)^+$ such that~$g(N\cap Z)\in (G/(N\cap Z))(\Q)$. Then~$gM_{x'}^{der}g^{-1}$ and~$gZ(M)_{x'}g^{-1}$ are defined over~$\Q$. Therefore~$gM_{x'}g^{-1}$ is defined over~$\Q$ and contains~$y:=g\cdot x'$ and thus~$M_y\leq gM_x'g^{-1}$. As~$gZ\in (G/Z)(\Q)$, the conjugation-by-$g$ isomorphism~$M_{x'}^{der}\to gM_{x'}^{der}g^{-1}$ induces an isomorphism of Shimura data
\[
\phi:
(M_{x'}^{ad},X^{ad}_{x'})\to ((gM_{x'}^{der}g^{-1})^{ad}, (gX_{x'})^{ad})
\]
which maps~$x'^{ad}$ to the image~$y'$ of~$y$ in~$(gX_{x'})^{ad}$. As~$x'^{ad}$ is Hodge generic in~$X^{ad}_{x'}$, so is~$y'$ in~$(gX_{x'})^{ad}$. We have thus
\[
((gM_{x'}^{der}g^{-1})^{ad}, (gX_{x'})^{ad},y')=(M^{ad}_y,X^{ad}_y,y^{ad}).
\]
The existence of the isomorphism~$\phi$ implies that~$y\in\Sigma_{X}(M_{x'},X_{x'},x')$. It follows~$y\in  \Sigma^g(x')$, and we deduce~\eqref{eq:geom in hybrid}.
\end{proof}
This concludes the proof of Theorem~\ref{thm:geom in hybrid}.
\end{proof}

\begin{definition}\label{def: orbite geo hybride}
The \emph{geometric hybrid orbit of~$x'$} is defined as
\[
\Sigma^{g}(x'):= \pi(W(\Q)^+).
\]
\end{definition}

Note that~$N\cap Z$ is reductive, and thus that~$W$ is an affine variety. We can consider global or local height function~$W(\Q)\to \Z_{\geq1}$. This may be helpful in defining height functions on the geometric Hecke orbit.
Together with the next Theorem, it would be a notion of height functions on the hybrid Hecke orbit.

\subsection{Finiteness Theorem}
The following is the analogue of~Theorem~\ref{thm:finiteness conjugacy special} 
\begin{theorem} \label{thm:finiteness geometric hybrid}
Let~$\Sigma_X(M,X_M,x)$ be a hybrid Hecke orbit and let us denote by~$P(\Sigma_X(M,X_M,x))$ the set of subsets of~$\Sigma_X(M,X_M,x)$.
Then
\[
\{\Sigma_g(x')|x'\in \Sigma_X(M,X_M,x)\}\subseteq P(\Sigma_X(M,X_M,x))\}
\]
is a finite set of subsets.
\end{theorem}
In other words, the hybrid orbit~$\Sigma_X(M,X_M,x)$ contains a finite subset~$\{y_1;\ldots;y_k\}\subseteq \Sigma_X(M,X_M,x)$ such that
\begin{equation}\label{eq:yis}
\forall x'\in \Sigma_X(M,X_M,x), \exists! y\in \{y_1;\ldots;y_k\}, \Sigma^g(x')=\Sigma^{g}(y).
\end{equation}
\begin{corollary}\label{cor:finiteness}
A hybrid Hecke orbit is a finite union of geometric hybrid orbits.
\end{corollary}

As~$W(\Q)$ is countable, the sets~$W(\Q)^+$ and~$\Sigma^g(x')$ are countable.
The following is thus a consequence of Corollary~\ref{cor:finiteness}.
\begin{corollary} A hybrid Hecke orbit is countable.
\end{corollary}
\begin{proof}[Proof of Theorem~\ref{thm:finiteness geometric hybrid}]Let~$\Sigma=\Sigma_X(M,X_M,x)$ be a hybrid Hecke orbit. For every~$x'\in\Sigma$, let~$(M_{x'},X_{x'})\leq (G,X)$ be 
the smallest subdatum such that~$x'\in X_{x'}$. There exists a unique morphism
\[
\phi_{x'}:
M^{ad}\to M^{ad}_{x'}
\]
which maps~$x^{ad}$ to~$x'^{ad}$. We consider the corresponding Lie algebra morphism
\[
\rho_{x'}:
\mathfrak{m}^{ad}\xrightarrow{d\phi_{x'}}
\mathfrak{m}^{ad}_{x'}\simeq \mathfrak{m}^{der}_{x'}\hookrightarrow\mathfrak{g}.
\]
We recall that, because~$\mathfrak{m}^{ad}$ and~$\mathfrak{g}$ are semisimple, there are finitely many~$G({\R})^+$-conjugation classes of Lie algebra morphisms
\[
\mathfrak{m}_{\R}^{ad}\to \mathfrak{g}_{\R}.
\]
It is enough to prove that for every~$\rho:\mathfrak{m}^{ad}\to \mathfrak{g}$, the set
\[
\Sigma_\rho:=\{x'\in \Sigma|\exists g\in G(\R)^+, \rho_{x'}=g\rho g^{-1}\}
\]
is a finite union of geometric hybrid orbits. We may assume that~$\Sigma_\rho\neq\emptyset$, and without loss of
generality we may assume~$\rho=\rho_y$
for some~$y\in \Sigma_X(M,X_M,x)$. Then~$M_y^{der}\leq G$ is the
semisimple $\Q$-subgroup with Lie algebra~$\mathfrak{m}_y^{der}=\rho(\mathfrak{m}^{ad})$. 
We consider~$Z:=Z_G(M_y^{der})$ and~$L:=M_y^{der}\cdot Z^0$ and~$F:=L/M_y^{der}$. Let~$F^{nc}\leq F^{der}(\R)$ be the product of the non compact~$\R$-quasi-factors of~$F$. 
The generic Mumford-Tate group~$H'$ of~$X_L:=L(\R)\cdot y=L^{nc}\cdot y$ satisfies~$y(\C^{\times})\leq L^{nc}\leq H'(\R)\leq L(\R)$, and~$(H',X_L)$ is a Shimura datum (\cite[Lem. 3.3]{U}). Let~$H:=H'/M_y^{der}$ and let~$(H',X_L)\to (H,X_H)$ be the corresponding quotient Shimura datum.
Then the image of~$(M_y,X_y)$ in~$(H,X_H)$ is a CM Shimura datum~$(M_y^{ab},y^{ab})$. 
Let~$T\leq F(\R)$ be a maximal torus such that~$T/Z(F)^0(\R)$ is compact, and let~$W_T=N_F(T)/Z_F(T)$ be the Weyl group of~$T$ in~$F$. There exists~$f\in F(\R)^+$ such that
\[
S:=fy^{ab}(\C^{\times})f^{-1}\leq fM^{ab}(\R)f^{-1}\leq T.
\]
Using the same argument as in the proof of Theorem~\ref{thm:finiteness conjugacy special} and the Remarks that follow, we see that there exists~$E\subseteq W_T$ such that
\[
fM^{ab}(\R)f^{-1}\text{ is generated by }\bigcup_{w\in E} w(S).
\]

We consider~$x'\in \Sigma_{\rho}$ and~$g\in G(\R)^+$ such that~$\rho_{x'}=g\rho g^{-1}$. 
We claim that there exists~$g'=gz\in gZ(\R)^+$ such that
\[
x'=g'\cdot y.
\]
\begin{proof}[Proof of the claim] We have~$M^{der}_{x'}=gM^{der}_{y}g^{-1}$ by assumption. Let~$K_{x'}:=Z_{M_x^{der}(\R)^+}(x'(\C^\times))\leq M_x^{der}(\R)^+$ and~$K_y:=Z_{M_y^{der}(\R)^+}(y(\C^\times))\leq M^{der}_{y}$ be the maximal compact subgroups associated with~$x'$ and~$y$.
Then we have~$K_{x'}=gK_{y}g^{-1}$. It follows that he Cartan involutions~$\Theta_{x'}=\Theta_{g\cdot y}:G\to G$ have the same restrictions~$\Theta_{x'}|_{M_{x'}^{der}},\Theta_{g\cdot y}|_{M_{x'}^{der}}:{M_{x'}^{der}}\to {M_{x'}^{der}}$. Let~$z'\in G(\R)^+$ be such that~$\Theta_{x'}=z'\Theta_{gy}{z'}^{-1}$. Then~$Ad_{z'}|_{M_{x'}^{der}}=Id_{M_{x'}^{der}}$, that is~$z'\in Z_G({M_{x'}^{der}})(\R)$. This implies the claim with~$z=g^{-1}z'g$.
\end{proof}

The conjugation by~$g'$ induces $\R$-algebraic isomorphisms
\[
Z\to Z':=Z_G(M_{x'}^{der})\quad
L\to L':=M_{x'}^{der}\cdot Z'^0\quad
F\to F':=L'/M_{x'}^{der},
\]
and induces a map
\[
y^{ab}\mapsto x'^{ab}
\]
and an isomorphism
\[
S=y^{ab}(\C^{\times})\to S':={x'}^{ab}(\C^{\times})
\]
and, with~$T'\leq F'$ being the torus corresponding to~$T$, an isomorphism
\[
W_T
\to 
W_{T'}:=N_{F'}(T')/Z_{F'}(T').
\]

For every~$E'\subseteq W_T$, let~$A_{E'}$ be the $\R$-torus generated by~$\bigcup_{w\in E'}w(S)$ and let~$M_{E'}$
be the inverse image of~$A_{E'}$ by the map~$L\to F$.
Then~$M_{E'}$ is a reductive group such that~$M^{der}_{E'}=\ker(L\to F)=M^{der}_y$.

There exists~$E'\subseteq W_T$ such that
\begin{equation}\label{eq..}
M_{x'}=g'M_{E'}{g'}^{-1}.
\end{equation}

Let~$x''\in\Sigma_\rho$ and~$g''\in G(\R)^+$ be such that that
\[
x''=g''\cdot y\text{ and }\rho_{x''}=g''\cdot \rho\cdot {g''}^{-1},
\]
and such that
\begin{equation}\label{eq...}
M_{x''}=g''M_{E'}{g''}^{-1}.
\end{equation}

We have, with~$\gamma=g''\cdot {g'}^{-1}$,
\[
M_{x''}=\gamma\cdot M_{x'}\cdot \gamma^{-1}\text{ and }x''=\gamma\cdot x'.
\]
Let~$W=G/(N_G(M_{x'}\cap Z_G(M_{x'}^{der}))$ be as in~\eqref{eq:W hybride}. We claim that
\[
\gamma\cdot (N_G(M_{x'})\cap Z_G(M_{x'}^{der}))\in W(\Q)^+.
\]
In terms of Definition~\ref{def: orbite geo hybride}, this formula is equivalent to~$x''\in\Sigma^{g}(x')$.
\begin{proof}From~$N_G(M_{x'})\cap Z_G(M_{x'}^{der})=N_G(Z(M_{x'}))\cap Z_G(M_{x'}^{der})$, we obtain an embedding
\[
W\to G/N_G(M_{x'})\times G/Z_G(M_{x'}^{der}).
\]
Note that~$\gamma\in G(\R)^+$.
It is enough to prove
\begin{equation}\label{eqratio1}
\gamma\cdot N_G(Z(M_{x'}))\in (G/N_G(Z(M_{x'})))(\Q)
\end{equation}
and
\begin{equation}\label{eqratio2}
\gamma\cdot Z_G(M_{x'}^{der})\in (G/Z_G(M_{x'}^{der}))(\Q).
\end{equation}
We have an isomorphism
\[
(G/Z_G(M_{x'}^{der}))\to G\cdot \rho_{x'}.
\]
As~$\rho_{x''}$ is defined over~$\Q$, we have~\eqref{eqratio2}.

Let~$G\cdot [Z(M_{x'})]$ denote the conjugacy class of the torus~$Z(M_{x'})$. We have an isomorphism
\[
(G/N_G(Z(M_{x'})))\simeq G\cdot [Z(M_{x'})].
\]
As~$\gamma\cdot Z(M_{x'})\gamma^{-1}=Z(M_{x''})$ and~$Z(M_{x''})$ is defined over~$\Q$, we have~\eqref{eqratio1}.
\end{proof}

For each~$E'\subseteq W_T$, all the~$x''\in \Sigma_\rho$ as in~\eqref{eq...} belong to the same hybrid geometric orbit. From the discussion preceding~\eqref{eq..}, for each~$x''\in\Sigma_\rho$, there is at least one~$E'\subseteq W_T$ which qualifies.

As there are only finitely many subsets~$E'\subseteq W_T$, we deduce that~$\Sigma_\rho$ is contained in a finite union of hybrid geometric orbits.

Let~$W_{x''}=N_G(Z(M_{x'}))\cap Z_G(M_{x'}^{der})$ and let~$\pi_{x''}:W_{x''}(\Q)\cap W_{x''}^+\to \Sigma^g(x'')\subseteq X$ be as in Definition~\ref{def: orbite geo hybride}. Conjugation by~$\gamma$ induces isomorphisms defined over~$\Q$
\[
G/N_G(M_{x'})\times G/Z_G(M_{x'}^{der})\to 
G/N_G(M_{x''})\times G/Z_G(M_{x''}^{der}),
\]
and thus~$W(\Q)^+\to W_{x''}(\Q)^+$. Moreover one can check that
\begin{multline}\label{eq:pi''}
g\cdot (N_G(Z(M_{x''}))\cap Z_G(M_{x''}^{der}))=
g\cdot\gamma\cdot (N_G(Z(M_{x''}))\cap Z_G(M_{x''}^{der}))\gamma^{-1}\\
\mapsto \pi\left( g \gamma \cdot (N_G(Z(M_{x'}))\cap Z_G(M_{x'}^{der})) \right)
\end{multline}
is~$G(\R)$ equivariant and maps~$1\cdot(N_G(Z(M_{x''}))\cap Z_G(M_{x''}^{der}))$ to
\begin{multline}\pi\left( 1\cdot \gamma \cdot (N_G(Z(M_{x'}))\cap Z_G(M_{x'}^{der})) \right)=\\
\gamma\cdot \pi\left( 1\cdot \gamma \cdot (N_G(Z(M_{x'}))\cap Z_G(M_{x'}^{der})) \right)=
\gamma \cdot x'=x''.
\end{multline}
The application~\eqref{eq:pi''} is thus~$\pi_{x''}$. We deduce
\[
\Sigma^g(x')=\pi(W(\Q)^+)=\pi_{x''}(W_{x''}(\Q)^+)=\Sigma^g(x'').
\]
As there are only finitely many subsets~$E'\subseteq W_T$,  this concludes the proof of Theorem~\ref{thm:finiteness geometric hybrid}.
\end{proof}

\section{Natural Height functions on Geometric hybrid orbits}\label{sec:height functions}
\emph{One of the main tools in~\cite{APZ1} is the introduction of natural height functions on geometric Hecke orbits. This
was done by using the parametrisation of the geometric Hecke orbit by a set~$W(\Q)^+$ of rational points of a conjugacy class~$W$ of~$G$ and using a height function on~$W$.}

\emph{
For hybrid geometric Hecke orbits, even of special points, the maps~\eqref{map special max param} are in general not injective on~$W(\Q)^+$. Given a point~$x\in\Sigma\cap X^+$ as in~\eqref{map special max param}, one first chooses a map~$p_i$ such as in~\eqref{param maps pi} such that~$x$ has one and only one inverse image on~$W_i(\Q)^+$. One can then choose a height function on~$W_i(\Q)$, and define the height of~$x$ as the height of this inverse image.}

\subsection{Uniqueness of preimages}
\begin{proposition}\label{prop:41} 
Let~$x'$ and~$W=G/(N\cap Z)$ be as in Theorem~\ref{thm:geom in hybrid}. Then
\[
\exists! w\in (W(\Q)\cap W^+), \pi(w)=x'.
\]
\end{proposition}
\begin{proof}The existence follows from the fact that~$\pi(w)=x'$ if~$w$ is the neutral coset~$w=N\cap Z\in W=G/(N\cap Z)$.

We need to prove the unicity. Let~$w\in W(\Q)\cap W^+$ be such that~$\pi(w)=x'$. As~$w\in W^+$, there exists~$g\in G(\R)^+$ such that~$w= g\cdot(N\cap Z)$. As~$\pi(w)=x'$, we have~$g\cdot x'=x'$. From~$g\cdot x'$ and~$g\in G(\R)^+$ we deduce
\[
g\in Z_G(x'(\C^\times)).
\]

As~$w\in W(\Q)$, that is~$g\cdot (N\cap Z)\in (G/(N\cap Z))(\Q)$, we have~
\[
gN_G(M_{x'})\in (G/N_G(M_{x'}))(\Q)
\text{ and~}gZ_G(M^{der}_{x'})\in (G/Z_G(M^{der}_{x'}))(\Q)
.\]

From~$gN_G(M_{x'})\in (G/N_G(M_{x'}))(\Q)$, we deduce that~$gM_{x'}g^{-1}$ is defined over~$\Q$. From~$g\in Z_G(x'(\C^\times))$ we deduce~$x'(\C^\times)=gx'(\C^\times)g^{-1}\leq gM_{x'}g^{-1}$. As~$gM_{x'}g^{-1}$ is defined over~$\Q$, we have~$M_{x'}\leq gM_{x'}g^{-1}$. We deduce~$M_{x'}= gM_{x'}g^{-1}$ and thus~$g\in N_{G(\R)^+}(M_{x'})$.

Consequently the conjugation by~$g$ defines automorphisms~$\phi:M_{x'}\to M_{x'}$, $\phi^{der}:M^{der}_{x'}\to M^{der}_{x'}$ and~$\phi^{ad}:M^{ad}_{x'}\to M^{ad}_{x'}$, defined over~$\R$.

From~$g\cdot Z_G(M^{der}_{x'})\in (G/Z)(\Q)$ we deduce that~$\phi^{der}$ is defined over~$\Q$. Thus~$\phi^{ad}$ is defined over~$\Q$, and the subgroup, say~$H\leq M^{ad}_{x'}$, of elements fixed by~$\phi^{ad}$, is defined over~$\Q$.
Let~$x'^{ad}:\C^{\times}\to M_{x'}(\R)\to M^{ad}_{x'}(\R)$ denote the cocharacter deduced from~$x'$.
From~$g\in Z_G(x'(\C^\times))$ we deduce that~$x'^{ad}(\C^{\times})\leq H(\R)$. As~$M_{x'}$ is the Mumford-Tate group of~$x'$, the Mumford-Tate group of~$x'^{ad}$ is~$M_{x'}^{ad}$. As~$H$ is defined ove~$\Q$ and contains~$x'^{ad}(\C)^{\times}$, we have~$M^{ad}_{x'}\leq H$. That is,~$\phi^{ad}$ is the identity automorphism. Consequently,~$\phi^{der}$ is trivial, that is~$g\in Z_{G}(M^{der}_{x'})$.

We proved~$g\in Z_{G}(M^{der}_{x'})$ and~$g\in N_G(M_{x'})$ we deduce~$g\in N\cap Z$ and~$w=1\cdot (N\cap Z)$.
This concludes the proof.
\end{proof}

\begin{proposition}\label{prop:unicity}
Let~$x'$ and~$W$ be as in Proposition~\ref{prop:41}. Let~$g\in G(\R)^+$ be such that~$g \cdot (N\cap Z)\in W(\Q)\cap  W^+$, and let~$x''=g\cdot x'$.

Assume that~$M_{x''}(\R)=gM_{x'}(\R)g^{-1}$. Then
\[
\exists! w\in (W(\Q)\cap W^+), \pi(w)=x''.
\]
\end{proposition}
\begin{proof}Taking~$w=g\cdot(N\cap Z)$ proves the existence.

Let~$N''=N_G(M_{x''})=gN_G(M_{x''})g^{-1}$, and~$Z''=Z_{G}(M_{x''})=gZ_{G}(M_{x'})g^{-1}$ and~$W''=G/(N''\cap Z'')$. Let~$\pi'':W''^+:=G(\R)^+/(N''\cap Z'')\to X$ be the unique~$G(\R)^+$ equivariant map such that~$\pi'':1\cdot N''\cap Z''\mapsto x''$.

Then~$h\cdot (N\cap Z)\mapsto h\cdot(N\cap Z)\cdot g^{-1}= hg^{-1}\cdot(N''\cap Z'')$ defines a~$\R$-algebraic $G(\R)$-equivariant isomorphism
\[
\phi:W\to W''
\]
We have~$\phi(W^+)=W''^+:=G(\R)^+\cdot g^{-1}\cdot(N''\cap Z'')$.
The map~$\pi''\circ \phi:W^+\to W''^+ \to X$ is~$G(\R)^+$-equivariant and
\[
\pi''\circ\phi(1\cdot (N\cap Z))=\pi''(g^{-1}\cdot (N''\cap Z''))=g^{-1}\cdot\pi''(1\cdot (N''\cap Z''))=g^{-1}\cdot x''=x'.\]
As~$\pi$ is the unique such map (Theorem~\ref{thm:geom in hybrid}), we have~$\pi''\circ \phi=\pi$.

By Lemma~\ref{lem:tout con}, the isomorphism~$\phi$ is defined over~$\Q$.

By Proposition~\ref{prop:41} applied to~$x''$ instead of~$x'$, the unique preimage of~$x''$ in~$W''(\Q)$ under~$\pi''$ is~$1\cdot (N''\cap Z'')$. Applying~$\phi^{-1}$ we deduce that the unique preimage of~$x''$ in~$W(\Q)$ under~$\pi$ is~$g\cdot (N\cap Z)$.
\end{proof}

\begin{lemma}\label{lem:tout con}
 Let~$H\leq G$ be  linear~$\Q$-algebraic groups, let~$g\in G(\C)$ be such that~$g\cdot H \in (G/H)(\Q)$.
Then~$H'=gHg^{-1}$ is a~$\Q$-algebraic subgroup of~$G$ and~$G/H'$ is a~$\Q$-algebraic variety. Moreover the isomorphism of algebraic varieties
\begin{equation}\label{eq:tout con}
fH\mapsto fHg^{-1}=fg^{-1}H':(G/H)(\C)\to (G/H')(\C)
\end{equation}
is defined over~$\Q$.
\end{lemma}
\begin{proof}For~$\sigma\in Aut(\C/\Q)$, we have~$\sigma(gH(\C))=gH(\C)$. It follows that~$\sigma(H(\C)g^{-1})=H(\C)g^{-1}$. Hence~\eqref{eq:tout con} maps~$\sigma(fH(\C))=\sigma(f)H(\C)$ to
\(
\sigma(f)H(\C)g^{-1}=\sigma(f)\sigma(H(\C)g^{-1})=\sigma(f H(\C)g^{-1}).
\)
As~\eqref{eq:tout con} commutes with the action of~$Aut(\C/\Q)$, it is defined over~$\Q$.
\end{proof}

\subsection{Natural heights functions}

\begin{proposition}\label{prop:factor W}
Let~$x'$ and~$W$ be as in Theorem~\ref{thm:geom in hybrid}. 

The varieties~$G/N$ and~$G/Z$ and~$\widetilde{W}:=G/Z_G(M_x')$ are affine.

We have~$Z_G(M_{x'})\leq N\cap Z$, we have~$(N\cap Z)^0=Z_G(M_{x'})^0$, and 
the map
\(
\widetilde{W}\to W
\)
is finite and étale.

The following map is a closed embedding:
\[
W\to G/N\times G/Z.
\]
\end{proposition}
We remark that~$\widetilde{W}$ in this proposition is the~$W$ which was used in~\cite{APZ1} to parametrise the geometric Hecke orbits, in the context of the André-Pink-Zannier Conjecture. Here,~$W$ parametrises the geometric hybrid orbits.

\begin{proof}
As~$G$ and~$M_{x'}$ are reductive so are~$Z_G(M_{x'})$ and~$N_G(M_{x'})$ and~$M^{der}_{x'}$ and~$Z=Z_G(M^{der}_{x'})$ and~$Z(M_{x'})^0$ and~$N=N_G(Z(M_{x'})^0)$. We deduce that~$G/N$ and~$G/Z$ and~$G/Z_G(M_{x'})$ are affine. The inclusion~$Z_G(M_{x'})\leq N\cap Z$ is immediate and the identity~$(N\cap Z)^0=Z_G(M_{x'})^0$ was proved in~\eqref{eq:NZvsZ}. It follows that~$\widetilde{W}\to W$ is finite and étale.

The map~$W \to G/N\times G/Z$ is obviously injective. Let us show that this is a closed map. As~$Z(M_{x'})^0$ is a torus, the image of~$N_G(Z(M_{x'})^0)\to \Aut(Z(M_{x'})^0)$ is discrete. Thus~$N_G(Z(M_{x'})^0)/Z_G(Z(M_{x'})^0)$ is finite and~$G/Z_G(M_{x'})\to G/(N\cap Z)$ is finite, and proper.

It is enough to prove that the map
\[
\widetilde{W}\to G/Z_G(Z(M_{x'})^0)\times G/Z_G(M^{der}_{x'})
\]
is a closed embedding. Let us denote by~$\phi':\mathfrak{z}(M_{x'})\to \mathfrak{g}$ and~$\phi':\mathfrak{m}^{der}_{x'}\to\mathfrak{g}$ and~$\phi=\phi'\oplus\phi'':\mathfrak{m}_{x'}\to \mathfrak{g}$ denote the embedding of Lie algebras, viewed as vectors~$\phi'\in V':=\mathfrak{z}(M_{x'})^\vee\tens\mathfrak{g}$, and~$\phi''\in V'':={\mathfrak{m}^{der}_{x'}}^\vee\tens\mathfrak{g}$ and~$\phi=\phi'\oplus\phi''\in V:= \mathfrak{m}_{x'}^{\vee}\tens \mathfrak{g}\simeq V'\oplus V''$. 
We view the vector spaces~$V$,~$V'$ and~$V''$ as representations through the adjoint representation of~$G$ on~$\mathfrak{g}$, and as affine algebraic varieties.

We have embeddings
\begin{align}
G/Z_G(Z(M_{x'})^0)\simeq G\cdot \phi'\subseteq V'\\
G/Z_G(M^{der}_{x'})\simeq G\cdot \phi''\subseteq V''\\
G/Z_G(M_{x'})\simeq G\cdot \phi\subseteq V.
\end{align}
By~\cite{Richardson} the orbits are closed. In particular
\[
G/Z_G(M_{x'})\subseteq G/Z_G(Z(M_{x'})^0)\times G/G/Z_G(M^{der}_{x'})\subseteq V'\oplus V''\simeq V
\]
which is closed as a subset of~$V$, is closed as a subset of~$G/Z_G(Z(M_{x'}))\times G/Z_G(M^{der}_{x'})$.
\end{proof}

\subsection{Explicit height functions}\label{sec:natural height}

Choosing a basis of~$\mathfrak{g}$ and~$\mathfrak{m}_{x'}$, resp. of~$\mathfrak{z}(M_{x'})$, resp. of~$\mathfrak{m}^{der}_{x'}$, we obtain identifications
\begin{align}
&V\phantom{''}\simeq \AAA^{\dim(G)\cdot \dim(M_{x'})},\\
\text{ resp. }
&V'\phantom{'}\simeq \AAA^{\dim(G)\cdot \dim(Z(M_{x'}))},\\
\text{ resp. }
&V''\simeq \AAA^{\dim(G)\cdot \dim(M^{der}_{x'})}.
\end{align}
We deduce global and local affine Weil height functions on~$\tilde{W}$, on~$G/Z_{G}(Z(M_{x'})^0)$ and~$G/Z$ respectively. By general functoriality of heights (Theorem~\ref{thm:func heights}), the height functions descend from~$G/Z_{G}(Z(M_{x'})^0)$ to~$G/N$ and from~$\tilde{W}$ to~$W$.

\subsection{Natural height functions}\label{sec:natural} We construct a natural height function on a hybrid orbit
\[
H_f:\Sigma_X(M,X_M,x)\to \Z_{\geq1}.
\]
For every~$x'\in \Sigma_X(M,X_M,x)$ we consider the geometric Hecke orbit~$\Sigma^{g}(x')\subseteq \Sigma_X(M,X_M,x)$ and the parametrisation
\[
\pi_{x'}:W_{x'}(\Q)^+\to \Sigma^g(x')
\]
where~$W_{x'}=G/(N_G(Z(M_{x'}))\cap Z_G(M^{der}_{x'})$. We choose~$y_1,\ldots,y_k$ as in~\eqref{eq:yis}.
We choose non-archimedean Weil height functions
\begin{equation}\label{eq:choices height}
H_{i,f}:W_{y_i}(\Q)\to \Z_{\geq1}.
\end{equation}
(Each~$H_{i,f}$ is unique up to polynomial equivalence by Theorem~\ref{thm:func heights}.) For every~$x'\in \Sigma_X(M,X_M,x)$ there is a unique~$y\in\{y_1;\ldots;y_k\}$ such that~$\Sigma^g(x')=\Sigma^g(y_i)$. By Proposition~\ref{prop:unicity},
there is a unique~$w(x')\in W_{y_i}(\Q)^+$ whose image by the parametrisation map~$\pi_{y_i}$ is~$x'$. 
\begin{definition} The natural height of~$x'$ is
\begin{equation}\label{eq:nat Hf}
H_f(x'):=H_{i,f}(w(x')).
\end{equation}
\end{definition}
Up to polynomial equivalence the  function~\eqref{eq:nat Hf} does depend on the choices made in~\eqref{eq:choices height}.

For every~$y_i$, we have maps
\[
w\to w^{cent}:W_{y_i}\to G/N_G(Z(M_{y_i}))
\qquad w\to w^{der}:W_{y_i}\to G/ Z_G(M^{der}_{y_i})
\]
We choose non-archimedean Weil height functions
\[
H_f:G/N_G(Z(M_{yi}))\to \Z_{\geq1} \qquad H_f:G/ Z_G(M^{der}_{y_i})\to \Z_{\geq1}
\]
we define
\begin{equation}\label{eq:defi Hfder Hfcent}
H^{cent}_f(x'):=H_f(w^{cent}(x'))\qquad
H^{der}_f(x'):=H_f(w^{der}(x')).
\end{equation}
It follows from Proposition~\ref{prop:factor W} that
\begin{equation}\label{eq:Hf vs cent et der}
H_f(x')\approx H_f(w^{cent}(x'))\cdot H_f(w^{der}(x')).
\end{equation}

\section{Height and discriminant of a linear torus}\label{sec:height disc}

We fix a reductive group~$G$ over~$\Q$ and an open compact subgroup~$K\leq G(\A_f)$.

For a torus~$T\leq G$ we consider:
\begin{itemize}
\item the splitting field of~$T$
\[
L=\left\{z\in \ol{\Q}\middle| \forall \sigma\in Gal(\ol{\Q}/\Q),\left( \sigma(z)=z\Leftrightarrow \forall y\in \Hom(GL(1)_{\ol{\Q}},T_{\ol{\Q}})), \sigma(y)=y\right)\right\},
\]
\item the absolute discriminant~$d_L$ of~$L$,
\item the unique maximal compact subgroup~$T_{max}\leq T(\A_f)$, 
\item the compact open subgroup~$T(\A_f)\cap K\leq T_{max}$.
\end{itemize}

In~\cite[cf.\,Def.\,1]{BiSYa} an important quantity is the ``discriminant"
\[
disc_K(T):=d_L\cdot [T_{max}:T(\A_f)\cap K].
\]
The main result of this section is the following. We use the notation~$\approx$ from~\cite[Def.\,1.7]{APZ1} to denote the polynomial equivalence of numerical functions.
\begin{theorem}[Height-equals-Discriminant Theorem]
\label{thm:AOvsHeight}
Let~$N_G(T)$ be the normaliser of~$T$, and let~$W\simeq G/N_G(T)$ the conjugacy class of~$T$. We denote by~$T_w\leq G_{\A_f}$ the torus corresponding to~$w\in W(\A_f)$. Let~$\iota:W\hookrightarrow \A^m$
be a closed affine $\Q$-algebraic embedding. 

Then, as functions~$W(\A_f)\to \Z_{\geq1}$,
\begin{equation}\label{Height equiv disc thm}
disc_K(T_w)\approx H_f(\iota(w)):=\max\{n\in\Z_{\geq1}| n\cdot \iota(w)\in \widehat{\Z}^m\}.
\end{equation}
\end{theorem}

Let~$G\leq GL(n)$ be an embedding defined over~$\Q$ and let~$G(\widehat{\Z}):=G(\A_f)\cap GL(n,\widehat{\Z})$.
Possibly conjugating the embedding by an element of~$GL(n,\Q)$, we may assume~$K\leq  G(\widehat{\Z})$.
Then
\[
\frac{ disc_{G(\widehat{\Z})}(T) }{disc_K(T)}\text{ belongs to~$\Z_{\geq 1}$ and divides }[G(\widehat{\Z}):K].
\]
It will thus be enough to prove Theorem~\ref{thm:AOvsHeight} in the case~$G=GL(n)$ and~$K=GL(n,\widehat{\Z})$.

Let us recall the principle of functoriality of heights.
We use the notation~$\succcurlyeq$ from~\cite[Def.\,1.7]{APZ1} to denote polynomial domination of numerical functions.
\begin{theorem}\label{thm:func heights}
Let~$W$ be an affine variety over~$\Q$, and let~$\iota_1:W\to \A^{m_1}$ and~$\iota_2:W\to \A^{m_2}$ be two morphisms such that~$\iota_1$ is a finite (proper) map. We define
\(
H_1(w):=\max\{n\in\Z_{\geq1}| n\cdot \iota_1(w)\in \widehat{\Z}^{m_1}\}\)
and \(
H_2(w):=\max\{n\in\Z_{\geq1}| n\cdot \iota_2(w)\in \widehat{\Z}^{m_2}\}
\)

Then, as functions on~$W(\ol{\Q}\tens\A_f)$,
\[
H_1
\succcurlyeq
H_2.
\]
\end{theorem}
In particular, if~$\iota_2$ is also a finite morphism, we have
\[
H_1\approx H_2.
\]

It follows that~\eqref{Height equiv disc thm} of Theorem~\ref{thm:AOvsHeight} is independent of the chosen affine embedding~$\iota:W\hookrightarrow \A^m$. We construct a natural embedding in~\S\ref{sec:lin torus}.

\subsection{Intrinsic discriminant of a Torus and its height}
\subsubsection{Canonical quadratic tensor}
For a torus over a field~$K$, we have a map
\begin{equation}\label{Y to Lie}
y\mapsto dy\mapsto dy (\partial):Y(T_K):=Hom(GL(1)_K,T)\to Hom(\mathfrak{gl}(1)_K,\mathfrak{t})\simeq \mathfrak{t},
\end{equation}
where we identify\footnote{By, say, choosing the generator~$\partial\in\mathfrak{gl}(1)_\Z$ characterised by~$d\chi(\partial)=1$ for the identity character~$\chi:GL(1)\to \A^1$.}~$\mathfrak{gl}(1)_\Z\simeq \Z$  and deduce a~$K$-linear map
\begin{equation}\label{Y to Lie tens K}
Y(T)\tens K\to \mathfrak{t}.
\end{equation}
Let us assume that~$K$ is of characteristic~$0$, and denote by~$\ol{K}$ an algebraically closed extension of~$K$.
Then~\eqref{Y to Lie tens K} is an injective map, which is bijective if~$T$ is split (for instance if~$K=\ol{K}$). 
We denote the image of~$Y(T_{\ol{K}})$ in~$\mathfrak{t}_{\ol{K}}$ by
\[
\mathfrak{y}\subseteq \mathfrak{t}_{\ol{K}}.
\]
Then~$\mathfrak{y}$ is a~$\Z$-structure of~$\mathfrak{t}_{\ol{K}}$. We deduce an isomorphism
\begin{equation}\label{iso det 2}
\det(\mathfrak{y})^{\tens2}\tens\ol{K}\simeq \det(\mathfrak{t}_{\ol{K}})^{\tens2}.
\end{equation}
Let~$y_1,\ldots,y_d$ be a $\Z$-basis of~$\mathfrak{y}\approx \Z^{\dim(T)}$. It is well defined up to the action of~$GL_\Z(Y(T))\approx GL(d,\Z)$. Then~$y_1\wedge\ldots\wedge y_d\in \det(Y(T_{\ol{K}}))$ is well defined up to~$\det(GL_\Z(Y(T)))\simeq \{+1;-1\}$,
and
\begin{equation}\label{canonical in det Y 2}
(y_1\wedge\ldots\wedge y_d)^{\tens2}\in \det(\mathfrak{y})^{\tens 2}
\end{equation}
does not depend on the choice of the basis~$y_1,\ldots,y_n$. We denote its image by~\eqref{iso det 2} by
\(
\eta_T\in \det(\mathfrak{t}_{\ol{K}})^{\tens2}.
\)
The map~$Y(T_{\ol{K}})\simeq \mathfrak{y}\subseteq \mathfrak{t}\tens_K\ol{K}$ is compatible with the action of~$Gal(\ol{K}/K)$. We deduce that~$\eta_T$ is fixed by the action of~$Gal(\ol{K}/K)$: we have
\[
\eta_T\in \det(\mathfrak{t})^{\tens2}.
\]
This is a canonical tensor on~$\mathfrak{t}$. As~$\dim(\det(\mathfrak{t}))=1$, we have~$\det(\mathfrak{t})^{\tens2}=\operatorname{Sym}^2\det(\mathfrak{t})$, and~$\eta_T$
represents a canonical quadratic form on~$\det(\mathfrak{t})$.

\subsubsection{A particular case}

Let~$E/\Q$ be a finite extension, and let~$T:=Res_{E/\Q}GL(1)$ be the torus such that~$T(\Q)=E^\times$.

The embeddings~$\iota_1,\ldots,\iota_{[E:\Q]}:E\to \ol{\Q}$ give morphisms~$E^\times \to \ol{\Q}^\times$ which are associated to characters~$y_1,\ldots,y_{[E:\Q]}:T_{\ol{\Q}}\to GL(1)_{\ol{\Q}}$.
Morever~$y_1,\ldots,y_{[E:\Q]}$ is a canonical basis of~$Y(T_{\ol{\Q}})$.

The \emph{Trace form} is a~$\Q$-bilinear form~$B_E:E\times E\to\Q$.
It induces a quadratic form~$\det B_E$ on the one dimensional~$\Q$-vector space~$\bigwedge^{[E:\Q]}E$, which correpond to an element, say~$\tau_E$, in~$\left(\bigwedge^{[E:\Q]}E\right)^{\tens2}$.

\begin{proposition} We have~$\tau_E=\eta_T$.
\end{proposition}
\begin{proof}
Let~$L\leq \ol{\Q}$ be a field over which~$T$ is split. Then
\[
(\iota_1\tens L,\ldots,\iota_{[E:\Q]}\tens L):E\tens L\to L\times\ldots\times L
\]
is an isomorphism of~$\Q$-algebras, identifying~$T\tens L\simeq GL(1)_L\times\ldots\times GL(1)_L$. The~$L$-linear extension of the Trace~$B_E$ form is, on~$L\times\ldots\times L$, the standard bilinear form
\[
B_E\tens L:\left((l_1,\ldots,l_{[E:\Q]}),(l'_1\ldots,l'_{[E:\Q]})\right)\mapsto
\sum_{i=1}^{[E:\Q]} l_i\cdot l'_i.
\]
As the canonical map~$Y(GL(1)_L)\to \mathfrak{gl}(1)_L\simeq L$ sends the identity embedding to~$1\in L$, the map~$Y(GL(1)_L)^{[E:\Q]}\simeq Y(T_{\ol{\Q}})\to \mathfrak{t}\tens L\simeq L^{[E:\Q]}$ is given by
\[
k_1\cdot y_1+\ldots +k_{[E:\Q]}\cdot y_{[E:\Q]}\mapsto (k_1,\ldots,k_{[E:\Q]})\in\Z^{[E:\Q]}\subseteq L^{[E:\Q]}.
\]
In particular, the basis~$y_1,\ldots,y_{[E:\Q]}\in Y(GL(1)^{[E:\Q]}_L)$, viewed in~$E\tens L\simeq  L\times\ldots\times L$ is the standard basis. It is thus orthonormal for the standard form~$B_E\tens L$. Consequently,~$y_1\wedge\ldots\wedge y_{[E:\Q]}$ is an orthonormal basis for the $L$-bilinear form~$\det (B_E\tens L)$ on~$\det(E)\tens L$. By definition of~$\eta_T$, we have
\[
\tau_E\tens L=\eta_T\tens L\text{ in }\Bigl({\bigwedge}^{[E:\Q]}E\Bigr)^{\tens2}\tens L,
\]
and thus~$\tau_E=\eta_T$.
\end{proof}
We denote by~$O_E\subseteq E$ the integral closure of~$\Z$ in~$E$.
\begin{proposition}\label{prop:lattice disc vs tenseur}
Let~$\Lambda\leq O_E$ be a sublattice, let~$\Lambda^\bot:=\{l\in L|B_E(l,\Lambda)\subseteq \Z\}$ be the~$B_E$-orthogonal of~$\Lambda$, and let
\begin{equation}\label{eq0:prop:lattice disc vs tenseur}
\Disc({\Lambda}):=[\Lambda^\bot:\Lambda]
\end{equation}
be its discriminant. Then, as lattices of~$\det(E)$,
\begin{equation}\label{eq1:prop:lattice disc vs tenseur}
\det({\Lambda})=\Disc(\Lambda)\cdot \det({\Lambda}^{\bot})
\end{equation}
and, as lattices of~$\det(E)^{\tens 2}$,
\begin{equation}\label{eq2:prop:lattice disc vs tenseur}
\det({\Lambda})^{\tens2}=\Disc(\Lambda)\cdot \eta_T\cdot \Z.
\end{equation}
\end{proposition}
\begin{proof}Recall that any lattice~$M$ and a sublattice~$N\leq M$, we have~$[\det(M):\det(N)]=[M:N]$.  Thus~\eqref{eq0:prop:lattice disc vs tenseur} implies~\eqref{eq1:prop:lattice disc vs tenseur}.

By definition,~$\det({\Lambda}^{\bot})$ is the orthogonal~$\det({\Lambda})^{\bot}$ of~$\det({\Lambda})$ with respect to the bilinear form~$\det(B_E)$ on~$\det(E)$.

Let~$\nu$ be a generator of~$\det({\Lambda})=\nu\cdot\Z$.
Then~$\xi=\nu/\sqrt{\Disc(\Lambda)}$ is an orthonormal basis of~$\det(B_E)$. Therefore~$\xi=\pm y_1\wedge\ldots\wedge y_{[E:\Q]}$, and thus~$\xi^{\tens2}=\eta_T$.
 Finally
\[
\det({\Lambda})^{\tens2}=\nu^{\tens2}\cdot \Z=\Disc(\Lambda)\cdot\xi^{\tens2}\cdot \Z=\Disc(\Lambda)\cdot \eta_T\cdot \Z.\qedhere
\]
\end{proof}
\subsection{Discriminant-height of an linearly embedded torus}\label{sec:lin torus}
For an injective~$\rho:T\to GL(n)_\Q$, we define
\begin{equation}\label{eq:canonical embedded tensor}
\eta_\rho :=\rho_\star \eta_T\in V_{\dim(T)}
\text{ where }V_d:=\Bigl(\bigwedge^{d}\mathfrak{gl}(n,\Q)\Bigr)^{\tens 2}.
\end{equation}
The standard basis of~$\mathfrak{gl}(n)$ induces a standard basis of~$\bigwedge^{d}\mathfrak{gl}(n)$ and thus of~$V_d$. This induces an  affine Weil height function on~$V_d$, and the non-archimedean part  is given in terms of the~$\Z$-structure~$V_{d,\Z}:=(\bigwedge^{d}\mathfrak{gl}(n,\Z))^{\tens2}$ by
\[
H_f(\eta):=\min \{n\in\Z_{\geq1}|n\cdot \eta\in V_{d,\Z}\}.
\]
We define the \emph{discriminant-height of~$\rho$} as
\[
\delta_\rho:=H_f(\eta_\rho):=\min \{n\in\Z_{\geq1}|n\cdot \eta_\rho\in V_{d,\Z}\}.
\]
We identify~$T$ with its image by~$\rho$.
\begin{proposition}As algebraic varieties, we have
\[
GL(n)/N_{GL(n)}(T)\simeq GL(n)\cdot \eta_{\rho}.
\]

Moreover~$GL(n)\cdot \eta_{\rho}$ is a closed subvariety of~$V_d$.
\end{proposition}
\begin{proof}We need to prove that~$N_{GL(n)}(T)$ is the stabiliser~$H$ of~$\eta_{\rho}$ in~$GL(n)$. We have~$\eta_{\rho}=\rho_\star\eta_{T}=\eta_{\rho(T)}=\eta_T$ and~$g\cdot \eta_{\rho}=\eta_{g\rho g^{-1}}=\eta_{gTg^{-1}}$. Thus~$\eta_\rho=g\cdot \eta_{\rho}$ whenever~$gTg^{-1}=T$. Thus
\[
N_G(T)\subseteq H.
\]

We denote by~$\mathfrak{t}\subseteq \mathfrak{gl}(n)$ the Lie algebra of~$T$.
We have~$\det(\mathfrak{t})\subseteq \bigwedge^{\dim(T)}\mathfrak{gl}(n)$.
As~$T$ is connected, we have
\[
N_{GL(n)}(T)=N_{GL(n)}(\mathfrak{t})=N_{GL(n)}(\det(\mathfrak{t})).
\]
The quadratic map~$x\mapsto x^{\tens2}$ induces an embedding, the Veronese embedding, from the projective space of~$\bigwedge^{\dim(T)}\mathfrak{gl}(n)$ to the projective space of~$\left(\bigwedge^{\dim(T)}\mathfrak{gl}(n)\right)^{\tens2}$. Thus,
\[
N_{GL(n)}(\det(\mathfrak{t}))=N_{GL(n)}(\det(\mathfrak{t})^{\tens2}).
\]
Note that~$\eta_\rho$ generates~$\det(\mathfrak{t})^{\tens2}$. Thus
\[
H\subseteq N_{GL(n)}(\det(\mathfrak{t})^{\tens2})=N_G(T).
\]
We have proven~$H=N_G(T)$ by double inclusion.

Let us prove the second assertion.
The centraliser of~$T$ in~$GL(n)$ is a Levi subgroup~$L$ of a parabolic subgroup~$P=N\cdot L\leq GL(n)$,
where~$N$, the radical of~$P$ is a unipotent group. By Kostant-Rosenlicht theorem \cite[Prop. 4.10]{BorelLAG}, the orbit~$N\cdot \eta_\rho$
is Zariski closed in~$V_d$. Observe that~$L\cdot \eta_\rho=\eta_\rho$. It follows that~$P\cdot \eta_\rho=N\cdot L\cdot \eta_\rho=N\cdot \eta_\rho$ is Zariski closed in~$V_d$.
Let~$K\leq G(\C)$ be a maxinal compact subgroup. Recall that~$G(\C)=K\cdot P(\C)$ (\cite[I.1.11, p.\,37]{BorelJi}). Since~$P(\C)\cdot \eta_\rho$ is closed and~$K$ is compact,
the set~$K\cdot P(\C)\cdot \eta_\rho$ is closed (for the archimedean topology). This implies that~$G(\C)\cdot \eta_\rho=K\cdot P(\C)\cdot \eta_\rho$
is Zariski closed in~$V_d$.
\end{proof}

\begin{corollary}On the conjugacy class~$W$ of~$T$, the map
\[
gTg^{-1}\mapsto \delta_{g\rho g^{-1}}
\]
is a finite part of an affine Weil Height function: that is, is of the form~$H_f(\iota(w))$ for an affine closed embedding~$W\to \A^m$. 
\end{corollary}

\subsubsection{A particular case}
Let~$E$ be a commutative semisimple subalgebra of~$\End(\Q^n)$, and~$T:=Res_{E/\Q}(GL(1))$.
We can write~$E\simeq L_1^{\oplus m_1}\oplus\ldots\oplus L_f^{\oplus m_f}$ for fields~$L_1,\ldots,L_f$, 
and we have accordingly~$T\simeq \prod Res_{L_i/\Q}(GL(1))^{m_i}$.

Let~$O_E$ be the integral closure of~$\Z$ in~$E$. We have~$O_E\simeq {O_{L_1}}^{\oplus m_1}\oplus\ldots\oplus {O_{L_f}}^{\oplus m_f}$. The trace form~$B_E$ on~$E$ is the direct sum~${B_{L_1}}^{\oplus m_1}\oplus\ldots\oplus {B_{L_f}}^{\oplus m_f}$ of the trace forms~$B_{L_i}$. We have thus
\[
d_E:=\Disc(O_E):=[{O_E}^\bot:O_E]=\prod [{O_{L_i}}^\bot:O_{L_i}]^{m_i}=\prod {d_{L_i}}^{m_i}.
\]

\begin{lemma}\label{lem:disc splitting}
There exists~$c:\Z_{\geq1}\to\Z_{\geq1}$ such that, if~$L$ denotes  the splitting field of~$E$,
\[
d_E\leq {d_L}^{c([E:\Q])}\text{ and }d_L\leq {d_E}^{c([E:\Q])}.
\]
Moreover,~$L$ is also the splitting field of~$T$.
\end{lemma}
\begin{proof}
For two number fields~$K,K'$, we have
\[
\max\{d_{K};d_{K'}\}\mid \operatorname{lcm}\{d_{K};d_{K'}\}\mid d_{K\cdot K'}\mid {d_{K}}^{[K':\Q]}\cdot {d_{K'}}^{[K:\Q]}.
\]
Thus~$K\leq K'$ implies~$d_{K}\mid d_{K'}$. If~$d_{L_m}=\max\{d_{L_i}\}$,
we deduce, for~$K=L_m$ and~$K'=L$
\[
d_E\leq {d_{L_m}}^{\sum m_i}\mid {d_{L}}^{\sum m_i}\mid{d_{L}}^{[E:\Q]}.
\]
Iterating~$K=L_1\cdot \ldots\cdot L_{i-1}$ and~$K'=L_i$, we deduce
\[
d_{L_1\cdot\ldots\cdot L_f}\leq (d_{L_1}\cdot \ldots \cdot d_{L_f})^{[L_1:\Q]\cdot \ldots\cdot [L_f:\Q]}\leq (d_E)^{[E:\Q]^f}.
\]
The Galois closure of~$L':=L_1\cdot\ldots\cdot L_f$ is~$L=L'_1\cdot \ldots L'_c$
where~$L'_i$ are the conjugates of~$L'$ and~$c\leq [L':\Q]!$ and~$[L':\Q]\leq \sum [L_i:\Q]\leq [E:\Q]$. We have
\[
d_{L}\leq (d_{L'_1}\cdot \ldots \cdot d_{L'_c})^{[L'_1:\Q]\cdot \ldots\cdot [L'_f:\Q]}= (d_{L'}^{c})^{[L':\Q]^c}.
\]
Finally~$d_{L}\leq {d_E}^{c([E:\Q])}$ with~$c([E:\Q])=[E:\Q]^{[E:\Q]+[E:\Q]!}\cdot [E:\Q]!$.
\end{proof}

Let~$\Lambda=E\cap \mathfrak{gl}(n,\Z)$. Then~$\Lambda$ is a lattice and a subalgebra, that is, is an \emph{order} in~$E$. We have
\[
\Disc(\Lambda):=[\Lambda^\bot:\Lambda]=d_E\cdot [O_E:\Lambda]^2.
\]

\begin{proposition}\label{prop:58}
We have, with~$\rho:T\to GL(n)$ the embedding induced by~$E\subseteq \End(\Q^n)$,
\[
\Disc(\Lambda)=\delta_\rho.
\]
\end{proposition}
We observe that if~$V\leq \Q^n$ is a~$\Q$-linear subspace, then~$\det(V)\cap \bigwedge^{\dim(V)}\Z^n=\det(V\cap \Z^n)$, and~$V^{\tens2}\cap (\Z^n)^{\tens2}=(V\cap \Z^n)^{\tens2}$.
Indeed, up to the action of~$GL(n,\Z)$, we may assume~$V=\Q^d\oplus 0\subseteq \Q^d\oplus\Q^{n-d}$,
where~$d=\dim(V)$.
\begin{proof}Proposition~\ref{prop:lattice disc vs tenseur} works for~$E$, even if~$E$ is not a field (that is, if~$f\geq2$). Together with the above observations (for~$E\leq \mathfrak{gl}(n,\Q)$ as~$V\leq \Q^{n^2}$),
\[
\Disc(\Lambda)\cdot \eta_T\cdot \Z=\det(\Lambda)^{\tens2}= V_{d,\Z}\cap \det(E)^{\tens 2}=V_{d,\Z}\cap \eta_T\cdot \Q.
\]
By definition,~$\delta_\rho=H_f(\eta_T)$ and
\[
V_{d,\Z}\cap \eta_T\cdot \Z=H_f(\eta_T)\cdot \eta_T\cdot\Z.
\]
we also have
\[
V_{d,\Z}\cap \eta_T\cdot \Q=\Disc(\Lambda)\cdot \eta_T\cdot \Z\subseteq V_{d,\Z}\cap \eta_T\cdot \Z
\subseteq V_{d,\Z}\cap \eta_T\cdot \Q.
\]
It follows~$\Disc(\Lambda)\cdot \eta_T\cdot \Z= V_{d,\Z}\cap \eta_T\cdot \Z$ and~$\delta_\rho=H_f(\eta_T)=\Disc(\Lambda)$.
\end{proof}

\subsection{Proof of the Height-equals-Discriminant Theorem}
Let~$T\leq G=GL(n)$ be as in Theorem~\ref{thm:AOvsHeight}. Let~$E=\Q[T(\Q)]\leq \mathfrak{gl}(n,\Q)$ be the associative subalgebra genereated by~$T$. Then~$E$ is a commutative semisimple subalgebra. We define~$T_E=Res_{E/\Q}GL(1)$ and~$\Lambda=E\cap \mathfrak{gl}(n,\Z)$. 

Recall that~$T$ has finitely many~$GL(n)$-conjugates contained in~$T_E$. We deduce that the map~$G/N_G(T)\to G/N_G(T_E):gTg^{-1}\mapsto gT_Eg^{-1}$ between conjugacy classes is a finite morphism of algebraic varieties.
By functoriality of heights, the map~$gTg^{-1}\mapsto gT_Eg^{-1}\mapsto \delta_{T_E}$ is polynomially equivalent to a height function of~$G/N_G(T)$.

Using Proposition~\ref{prop:58}, and substituting the right-hand side of~\eqref{Height equiv disc thm}, we are to prove that, as~$T$ varies in a geometric conjugacy class,
\begin{equation}\label{proof:eq1}
d_L\cdot [T_{max}:T(\A_f)\cap K]\approx d_E\cdot [O_E:\Lambda].
\end{equation}

Arguing as in~\cite[Lem.~7.2]{Tsi} (denoting~$e_{[E:\Q]}$ the~$d_g$ of loc. cit.) we have
\[
[O_E:\Lambda]\leq 
[T_{max}:T(\A_f)\cap K]^{c_{[E;\Q]}}{d_E}^{e_{[E:\Q]}}.
\]
This proves, with~$a=\max\{c_{[E;\Q]};1+e_{[E:\Q]}\}$,
\begin{equation}\label{proof:eq2}
(d_E\cdot [T_{max}:T(\A_f)\cap K])^a\geq  d_E\cdot [O_E:\Lambda].
\end{equation}

Let us define~$\widehat{O_E}:=O_E\tens\widehat{\Z}$ and~$\widehat{\Lambda}=\Lambda\tens\widehat{\Z}$. Observe that~$[\widehat{O_E}:\widehat{\Lambda}]=[O_E:\Lambda]$.

Let~$\Gamma=(1+\wh{\Lambda})\cap \widehat{O_E}^\times$ denote the congruence subgroup associated with~$\wh{\Lambda}$.

We also have~$T_{max}\leq \widehat{O_E}^\times$ and
\[
T(\A_f)\cap K=T(\A_f)\cap GL(n,\widehat{\Z})= T(\A_f)\cap GL(n,\widehat{\Z})\cap \widehat{O_E}^\times=T(\A_f)\cap\Gamma.\]
Thus
\[
[T_{max}:T(\A_f)\cap K]\leq [ \widehat{O_E}^\times:\Gamma].
\]
As~$\Gamma$ is the stabiliser of the coset~$1+\widehat{\Lambda}\in \widehat{O_E}/\widehat{\Lambda}$, we have
\[
\widehat{O_E}^\times/\Gamma\simeq \widehat{O_E}^\times\cdot \{1+\widehat{\Lambda}\}\subseteq \widehat{O_E}/\widehat{\Lambda}.
\]
Thus
\[
[ \widehat{O_E}^\times:\widehat{\Lambda}^\times]\leq [\widehat{O_E}:\widehat{\Lambda}]= [{O_E}:{\Lambda}].
\]
Finally
\begin{equation}\label{proof:eq3}
d_E\cdot [T_{max}:T(\A_f)\cap K]\leq d_E\cdot [ \widehat{O_E}^\times:\widehat{\Lambda}^\times]\leq d_E\cdot [{O_E}:{\Lambda}]
\end{equation}

By~\eqref{proof:eq2} and~\eqref{proof:eq3}, as~$T$ varies in its conjugacy class,
\[
d_E\cdot [T_{max}:T(\A_f)\cap K]\approx d_E\cdot [O_E:\Lambda].
\]
By Lemma~\ref{lem:disc splitting}  we may substitute~$d_L$ with~$d_E$. This proves~\eqref{proof:eq1}.

\section{Large Galois orbits conjecture.}\label{sec:Galois conj}
 A fundamental difficulty when implementing the Pila-Zanier strategy is to obtain appropriate lower bounds on the size of Galois orbits. We conjecture the following.
\begin{conjecture}\label{conj:Bounds}
Let~$S$ be a Shimura variety, let~$F$ be a field of finite type over the reflex field of~$S$, and let~$s=[x,1]\in S(F)$. 

Then, as~$x'$ varies in the geometric hybrid orbit~$\Sigma^g(x)$,
\[
[F([x',1]):F]\approx H_f(x').
\]
where~$H_f$ denotes the natural height function as in~\eqref{eq:nat Hf}.
\end{conjecture}
Here is the main technical result of this article.
\begin{theorem}\label{thm:tech}
Conjecture~\ref{conj:Bounds} holds true if~$S$ is of abelian type.
\end{theorem}
\subsection{Dichotomy strategy}\label{sec:dicho}
To make our strategy precise, we introduce some notations. 

Let~$S$ be a Shimura variety, and let~$F/E$ be a finite type extension of the reflex field of~$S$. We consider a geometric hybrid orbit~$\Sigma^g(s)\subseteq S(\ol{K})$ of a point~$s=[x,1]\in S(F)$, and the parametrisation
\[
W(\Q)^+\to \Sigma^g(s)
\]
and the embedding~$w\mapsto (w^{cent},w^{der}):W\to G/N\times G/Z$.

We can reformulate Conjecture~\ref{conj:Bounds}.
\begin{conjecture}\label{conj:Bounds bis}.
As~$x'$ varies in the geometric hybrid orbit~$\Sigma^g(x)$,
\begin{equation}\label{conjW:eq}
[F([x',1]):F]\succcurlyeq H_f(w(x')).
\end{equation}
\end{conjecture}
Note that~$H_f(w)\approx \max\{H_f(w^{cent});H_f(w^{der})\}$.
\begin{lemma}Let~$\Xi\subseteq\Sigma^g(x)$ and define~$\gamma(x'):=[F([x',1]):F]$. Assume that, as functions of~$x'\in\Xi$, at least one of the following holds true.
\begin{align}
\label{Route 1}
\gamma(x')\succcurlyeq H_f(w(x')^{cent})\text{ and }H_f(w(x')^{cent})\cdot \gamma(w)\succcurlyeq H_f(w(x')^{der}).\\
\label{Route 2}
\gamma(x')\succcurlyeq H_f(w(x')^{der})\text{ and }H_f(w(x')^{der})\cdot\gamma(w)\succcurlyeq H_f(w(x')^{cent}).
\end{align}
Then, as functions of~$w\in\Xi$,
\begin{equation}\label{eq:42}
\gamma(x')\succcurlyeq H_f(w(x')).
\end{equation}
\end{lemma}
\begin{proof}We apply the following. For~$f,g,h:\Xi\to \R_{\geq1}$, we have
\begin{multline}\label{eq:dichotomy trick}
fg\succcurlyeq h\text{ and }  f\succcurlyeq{g}
\Rightarrow
\\
f\approx f^2\succcurlyeq fg\succcurlyeq h\text{ and }  f\succcurlyeq{g}
\Rightarrow
\\
f\succcurlyeq \max\{g;h\}\geq \sqrt{g\cdot h}\approx g\cdot h.\qedhere
\end{multline}
\end{proof}

%

\section{Consequences of the Uniform Integral Tate property}\label{sec:Tate}

We consider a Shimura datum~$(M,X_M)$ and a Hodge generic~$x\in X_M$ and an extension~$F/E(M^{ad},X^{ad}_M)$ such that~$[x^{ad},1]\in Sh_K(M^{ad},X^{ad}_M)(F)$ for some open compact subgroup~$K\leq M^{ad}(\A_f)$.
Let~$(x,g)\mapsto [x,g]$ denote the map from~$\{x\}\times M(\A_f)$ to the infinite level Shimura variety~$Sh_{}(M,X_M)$. The group~$\frac{M(\A_f)}{Z(M)(\Q)}$ acts on the right on~$Sh_{}(M,X_M)$. There is a unique representation
\begin{equation}\label{eq:defi rhox}
\rho_x:Gal(\ol{F}/F\cdot E(M,X_M))\to \frac{M(\A_f)}{Z(M)(\Q)},
\end{equation}
such that
\[
\forall \sigma\in Gal(\ol{F}/F\cdot E(M,X_M)), \sigma([x,1])=[x,1]\cdot \rho_x(\sigma).
\] 

We denote by
\begin{equation}\label{eq:defi Uxf}
U_{x,F} \leq \frac{M(\A_f)}{Z(M)(\Q)}
\end{equation}
 the image of~$\mathcal{G}_{x,F}:=Gal(\ol{F}/F\cdot E(M,X_M))$.

The integral uniform Tate property is defined in~\cite[Def.~2.1]{APZ2}. Recall that if~$(M,X)$ is of abelian type and~$F$ is of finite type over~$\Q$, then~$U_{x,F}$ satisfies the integral uniform Tate property. 

The following extends a result of Serre on a conjecture of Lang~\cite{SerreOEuvres4}.
\begin{theorem}\label{thm:puissancecentre}
Assume that~$U_{x, F}$ satisfies the integral uniform Tate property. Let~$(G,X_G)$ be such that~$(M,X)$ is a subdatum of~$(G,X_G)$. Then there exists~$h\in\Z_{\geq1}$, depending only on~$x^{ad}$ and~$F$ and~$(G,X_G)$, such that
\[
\forall z \in Z(M)(\A_f),
z^h\cdot {Z(M)(\Q)}\in U_x.
\]

\end{theorem}
We consider the natural maps
\[
(\pi_{ad},\pi_{ab}):
\frac{M(\A_f)}{Z(M)(\Q)}\to M^{ad}(\A_f)\times \frac{M^{ab}(\A_f)}{M^{ab}(\Q)}.
\]
and denote  by~$(\ol{\pi_{ad}},\ol{\pi_{ab}}):M^{ad}(\A_f)\times \frac{M^{ab}(\A_f)}{M^{ab}(\Q)}\to M^{ad}(\A_f)\times \frac{M^{ab}(\A_f)}{M^{ab}(\Q)}$ the identity map.
Let
\[
\Gamma_{x,F}\leq \frac{M(\A_f)}{Z(M)(\Q)}
\]
 be the image of~$U_{x,F}$.
We define~$Y_{x,F}:=\ol{\pi_{ad}}(\Gamma_{x,F})=\pi_{ad}(U_{x,F})$. We note that~$[\mathcal{G}_{x^{ad},F}:\mathcal{G}_{x,F}]\leq [E(M,X): E(M^{ad},X^{ad})]\leq [E(M^{ab},X^{ab}):\Q]\leq i(\dim(M^{ab}))$ 
for some universal function~$i:\Z_{\geq1}\to\Z_{\geq1}$. We have~$[Y_{x^{ad},F}:Y_{x,F}]\leq i(\dim(M^{ab}))$.

\begin{theorem}\label{thm:Uestproduit}
Assume that~$U_x$ satisfies the integral uniform Tate property. Then there exists~$e,c\in\Z_{\geq1}$ and
\begin{equation}\label{Goursat}
Y'_{x^{ad},F}\subseteq Y_{x^{ad},F},
\end{equation}
 depending only on~$x^{ad}$ and~$F$, 
such that~
\begin{equation}\label{eqgoursat1}
\forall y\in Y_{x^{ad},F}, y^e\in Y'_{x^{ad},F},
\end{equation}
and that, denoting~$\Gamma'_x$ the inverse image of~$Y'_{x^{ad}}$ in~$\Gamma_x$,
we have
\begin{equation}\label{Goursat Gamma}
\left[\Gamma'_x:\left(\Gamma'_x\cap M^{ad}(\A_f)\right)\cdot \left(\Gamma'_x\cap \frac{M^{ab}(\A_f)}{M^{ab}(\Q)}\right)\right]\leq c.
\end{equation}
and, denoting by~$U'_x$ the inverse image of~$\Gamma'_x$ in~$U_x$, we have
\begin{equation}\label{Uprimisproduct}
\left[U'_x:\left(U'_x\cap \frac{M^{der}(\A_f)}{Z(M^{der})(\Q)}\right)\cdot \left(U'_x\cap \frac{Z(M)(\A_f)}{Z(M)(\Q)}\right)\right]\leq c \cdot g,
\end{equation}
where~$g:=\#\ker\left(H^1(\Q;Z(M^{der}))\to H^1(\A_f;Z(M^{der})\right)<+\infty$. 

Finally
\begin{equation}\label{thm:goursat:eq:exposant}
\forall \gamma\in \Gamma_x,\gamma^e\in \Gamma'_x\qquad\forall \gamma\in U_x,\gamma^e\in U'_x
\end{equation}
\end{theorem}

For a prime~$p$, and a compact group~$K$, we denote by~$K^\dagger\leq K$ the subgroup generated by the set~$\{k\in K|\lim_{i\to \infty} k^{p^i}\to 1\}$ of topologically~$p$-nilpotents elements of~$K$.

\begin{theorem}\label{thm:complement}
We consider the setting of Theorem~\ref{thm:Uestproduit}.

There exists~$l\in\Z_{\geq1}$ and 
\[
V=\prod_p V_p\leq M^{der}(\A_f)
\]
with~$l$ and the~$V_p\leq M^{ad}(\Q_p)$ depending only on~$x^{ad}$ and~$F$ such that~$V_p^\dagger=V_p$ for every~$p$ and
\[
V \leq U_x\cap M^{der}(\A_f)
\]
and
\[
\forall u\in U_x\cap M^{der}(\A_f), u^l\in V.
\]
\end{theorem}

\subsection{Proof of Theorem~\ref{thm:puissancecentre} assuming Theorem~\ref{thm:Uestproduit}}	Let us recall the following.

\begin{theorem}[{Corollary of~\cite[Th.\,3.3 and its proof]{CU}}]
\label{thm:cor CU}
There exists~$c:\Z_{\geq1}\to\Z_{\geq1}$ satisfying the following.

Let~$\phi:R\to T$ be an epimorphism of $\R$-anisotropic $\Q$-algebraic tori.
Let~$f=\#\pi_0(U)$ be the number of geometric component of~$U:=\ker(\phi)$.
Then
\[
\forall p,\#Coker(R(\Z_p)\to T(\Z_p))~\Bigl\vert~ e:=c(\dim(T))\cdot f.
\]
We have
\begin{equation}\label{eq:thm:cor CU}
\phi(R(\wh{\Z}))\geq \Theta_e(T)=\{t^e|t\in T_m\}.
\end{equation}
\end{theorem}
\begin{proof}[Proof of Theorem~\ref{thm:cor CU}]
By~\cite[Th.\,3.3 and its proof]{CU},  the cardinalily of
\[
Coker\left(R(\wh{\Z})\to (R/U^0)(\widehat{\Z})\right)
\]
is uniformly bounded by some~$c'(\dim(R))$. It that~$\Theta_{c}(R/U^0)$ is in the image of~$R(\wh{\Z})$ in~$(R/U^0)(\wh{\Z})$, where~$c=c(\dim(R)):=c'(\dim(R))!$. We may thus replace~$R$ by~$R/U^0$. There exists a factorisation
\[
T\xrightarrow{\phi'}R\xrightarrow{\phi}T\text{ such that }\phi\circ \phi'
\]
is the multiplication by~$e$, where~$e=\#\pi_0(\ker{\phi})$. We may thus assume that~$\phi=e$ and~$R=T$. We have then~$\phi(R(\wh{\Z}))=\{t^e|t\in T_m\}$.
\end{proof}
We recall the following
\begin{theorem} [Borel]\label{thm:Borel}
Let~$\phi: T\to R$ be a morphism of $\Q$-algebraic tori
 and~$S$ be its kernel. Then there exist~$c(S)$ such that

Then
\[
[\phi(T(\A_f))\cap R(\Q):\phi(T(\Q))]\leq c(S)<+\infty
\]
\end{theorem}
\begin{proof}[Proof of Theorem~\ref{thm:Borel}]
Galois cohomology gives us morphisms
\[
\delta:R(\Q)/\phi(T(\Q))\hookrightarrow H^1(\Q;S)
\]
and
\[
\delta_p:R(\Q_p)/\phi(T(\Q_p))\hookrightarrow H^1(\Q_p;S_{\Q_p}).
\]
For~$t\in \phi(T(\A_f))$, we have
\[
\forall p, \delta_p(t)=0.
\]
It follows
\[
\left(\phi(T(\A_f))\cap R(\Q)\right)/\phi(T(\Q))\hookrightarrow \ker\left(H^1(\Q;S)\to \prod_p H^1(\Q_p;S)\right).
\]
The right-hand side is finite, by Borel's~\cite[III.\S4.6 Théorème 7']{LNM5}, and does only depend on~$S$ by construction.
\end{proof}

\begin{proof}[Proof of Theorem~\ref{thm:puissancecentre} assuming Theorem~\ref{thm:Uestproduit}]
We observe that the map
\[
Gal(\ol{F}/F\cdot E(M,X_M))\xrightarrow{\rho_x}\frac{M(\A_f)}{Z(M)(\Q)}\to \frac{M^{ab}(\A_f)}{M^{ab}(\Q)}
\]
is the restriction to~$Gal(\ol{F}/F\cdot E(M,X_M))\leq Aut(\ol{F}/E(M^{ab},\{x^{ab}\}))$
of the reciprocity map
\[
{rec}_{x^{ab}}:Aut(\ol{F}/E(M^{ab},\{x^{ab}\}))\to Gal(\ol{\Q}/E(M^{ab},\{x^{ab}\}))\to \frac{M^{ab}(\A_f)}{M^{ab}(\Q)}.
\]
Let~$R=Res_{E(M^{ab},\{x^{ab}\})/\Q}GL(1)$ and let~$\phi:R\to M^{ab}$ be the reflex norm. By the Deligne-Shimura reciprocity law, the image of~${rec}_{x^{ab}}$
is~$\phi(R(\A_f))/M^{ab}(\Q)$. 

We deduce, with~$d:=[(F\cap \ol{\Q}):\Q]\cdot [E(M^{ad},X^{ad}):\Q]$,
\[
[\phi(R(\A_f))/M^{ab}(\Q):\pi_{ab}(U_x)]\leq [(F\cap \ol{\Q})\cdot E(M,X):E(M^{ab},\{x^{ab}\})]\leq d.
\]
We recall that~$\#\pi_0(\ker(\phi))\leq b$ with~$b=b(G,X_G)$ depending only on~$(G,X_G)$. We apply Theorem~\ref{thm:cor CU}. 
We deduce that
\[
\phi(R(\A_f))/M^{ab}(\Q)\geq \{t^e\cdot M^{ab}(\Q) | t\in M^{ab}(\A_f)\}.
\]
with~$f=b\cdot c(\dim(T))$ bounded in terms of~$(G,X_G)$.
We have thus
\[
\pi_{ab}(U_x)\geq \{t^{f\cdot d}\cdot M^{ab}(\Q) | t\in M^{ab}(\A_f)\}.
\]

Let~$U''_x=\left(U'_x\cap M^{der}(\A_f)\right)\cdot \left(U'_x\cap \frac{Z(M)(\A_f)}{Z(M)(\Q)}\right)$.

Using~\eqref{thm:goursat:eq:exposant} and~\eqref{Uprimisproduct}, we deduce
\begin{multline}
\pi_{ab}\left(U'_x\cap \frac{Z(M)(\A_f)}{Z(M)(\Q)}\right)=
\pi_{ab}(U''_x)\geq \{t^{f\cdot d\cdot e\cdot c\cdot g}\cdot M^{ab}(\Q) | t\in M^{ab}(\A_f)\}\\
\geq \{t^{f\cdot d\cdot e\cdot c\cdot g}| t\in
\pi_{ab}(Z(M)(\A_f)/Z(M)(\Q))\}
\end{multline}
Let~$h$ be the degree of the isogeny~$Z(M)\to M^{ab}$. Then the kernel of
\[
Z(M)(\A_f)/Z(M)(\Q)\to M^{ab}/ab(Z(M)(\Q))\
\]
is killed by~$h$ and the kernel of
\[
M^{ab}(\A_f)/ab(Z(M)(\Q))\to M^{ab}(\A_f)/M^{ab}(\Q)
\]
is of size~$[M^{ab}(\Q)\cap ab(Z(M)(\A_f):ab(Z(M)(\Q))]$, which divides~$g$ by Th.~\ref{thm:Borel}. There exists thus a morphism~$\beta$ such that
\[
Z(M)(\A_f)/Z(M)(\Q)\to 
\pi_{ab}(Z(M)(\A_f)/Z(M)(\Q))
\xrightarrow{\beta} Z(M)(\A_f)/Z(M)(\Q)
\]
is the multiplication by~$m=g\cdot h$. 

We get
\begin{multline}
U_x\geq 
U'_x\cap \frac{Z(M)(\A_f)}{Z(M)(\Q)}\geq 
\beta\circ \pi_{ab}\left(U'_x\cap \frac{Z(M)(\A_f)}{Z(M)(\Q)}\right)\\
\geq 
\beta(\{t^{f\cdot d\cdot e\cdot c\cdot g}| t\in
\pi_{ab}(Z(M)(\A_f)/Z(M)(\Q))\})\\
=
\{t^{f\cdot d\cdot e\cdot c\cdot g\cdot m}| t\in
Z(M)(\A_f)/Z(M)(\Q))\}\\
\geq 
\{t^{f\cdot d\cdot e\cdot c\cdot g\cdot m}\cdot Z(M)(\Q))| t\in
Z(M)(\A_f)\}\qedhere
\end{multline}
\end{proof}
\subsection{Proof of Theorem~\ref{thm:Uestproduit}}\label{sec:Proof of Theorem Uestproduit}

We prove Theorem~\ref{thm:Uestproduit}. Let~$Y'_{x, F}\leq Y_{x, F} $ be the subgroup generated by commutators.
We first prove that
\begin{equation}\label{bounded abelianised}
[Y_x:Y'_{x,F}]\leq c(F,x^{ad})<\infty.
\end{equation}
\begin{proof}[Proof of~\eqref{bounded abelianised}]
We use the results of~\cite[Appendix B]{APZ2}. We consider an embedding~$M^{ad}\to GL(N)$. Conjugating by some element~$g\in GL(N,\Q)$, we may 
assume~$U_{x^{ad}}\leq M^{ad}(\wh{\Z}):=M^{ad}(\A_f)\cap GL(N,\wh{\Z})$.

By assumption~$U_x$ satisfies the uniform integral Tate conjecture. It follows from~\cite[Prop. 4.3]{APZ2} that~$U_{x^{ad}}$ satisfies the uniform integral Tate conjecture.
We consider the context of~\cite[Th.~B.1 and \S B.3]{APZ2} using~$U_{x^{ad}}\leq M^{ad}(\wh{\Z})$ for what is denoted~$U\leq M(\wh{\Z})$ in loc. cit.

We consider~$U'\leq U$ from~\cite[\S B.3]{APZ2} given by~\cite[Th.~B.1]{APZ2}. We may choose~$p_0$ in~\cite[(109)]{APZ2} arbitrarily large and assume~$p_0\geq i(\dim(M^{ab}))$.
Here,~$M=M^{ad}$ and thus~$U=Y$ and~$U(R)=\prod_{p\geq p_0} V(p)^{\dagger}$. 

We consider~$Y(p)$ as in~\cite[B.3.1]{APZ2} and denote by~$W(p)$ the derived subgroup of~$Y(p)$. Then the image of~$Y(p)\leq GL(N,\Z_p)$ in~$GL(N,\F_p)$ is~$V(p)^\dagger$ .
We deduce that~$Y(p)$ is generated by topologically~$p$-nilpotent elements. The image of~$Y_x$ in~$Y_{x^{ad}}$ is of index at most~$i(\dim(M^{ab}))$. Thus, for~$p\geq i(\dim(M^{ab}))$
the image of~$Y_x$ in~$Y(p)$ is~$Y(p)$.

From~\cite[\S B.3, (110) (111)]{APZ2} we get
\[
\forall p\geq p_0, [Y(p):W(p)]<+\infty\text{ and }\forall p\gg0, [Y(p):W(p)]=1.
\]

We consider~$p\leq p_0$. Let~$Y_{x}(p)$ be the image of~$Y_{x,F}$ in~$Y(p)$.
Let~$Z(p)\leq Y_{x^{ad}}(p)$ be the intersection of the subgroups of index at most~$i(\dim(M^{ab}))$. Then~$Z(p)\leq Y_{x,F}(p)\leq Y_{x^{ad},F}(p)$.
Recall that we have~$[Y_{x^{ad},F}(p):Z(p)]<+\infty$. There exists thus a finite extension~$F'/F$ such that~$Z(p)=Y_{x^{ad},F'}(p)$ for all~$p\leq p_0$.

Let us denote by~$Y_x(p)'\leq Y_x(p)$ and~$Y_{x^{ad},F'}(p)'=Z(p)'\leq Z(p)$ the derived subgroups. Then~$[Y_x(p):Y_x(p)']\leq [Y_x(p): Z(p)]\cdot [Z(p): Z(p)']$ and~$[Z(p): Z'(p)]=[F':F]$.
Arguing as above, we can apply~\cite[\S B.3, (110) (111)]{APZ2} and get
\[
\forall p, [Y_x(p):Y_x(p)']\leq [Y_{x^{ad},F'}(p),Y_{x^{ad},F'}(p)']\cdot[F':F]<+\infty.
\]
Let~$Y'_x\leq Y_x$ be the derived subgroup. We have
\begin{multline}
[Y_x:Y_x']\leq \prod_p [Y_x(p):Y_x(p)']\leq \\
\left(\prod_{p\leq p_0} [Y_{x^{ad},F'}(p),Y_{x^{ad},F'}(p)']\right)\cdot[F':F]\cdot \prod_{p>p_0} [Y(p):W(p)]<+\infty.
\end{multline}

We deduce~\eqref{bounded abelianised}
\end{proof}
\begin{proof}[Proof of~\eqref{Goursat Gamma}]
We apply Goursat's Lemma as in~\cite[\S B.3.2]{APZ2} with the following modifications. We define~$G_1=M(\A_f)$ and~$G_2=M^{ab}(\A_f)/M^{ab}(\Q)$. According to Goursat's Lemma there are isomorphisms
\[
\frac{\ol{\pi_{ad}}(\Gamma'_x)}{\Gamma'_x\cap M^{ad}(\A_f)}
\simeq 
\frac{\Gamma'_x}{(\Gamma'_x\cap M^{ad}(\A_f))\cdot(\Gamma'_x\cap \frac{M^{ab}(\A_f)}{M^{ab}(\Q)})}
\simeq
\frac{\ol{\pi_{ab}}(\Gamma'_x)}{\Gamma'_x\cap \frac{M^{ab}(\A_f)}{M^{ab}(\Q)}}.
\]
As the right-hand-side is abelian, we have~$\Gamma'_x\cap M^{ad}(\A_f)\supseteq Y'_x$. We deduce~\eqref{Goursat Gamma} with~$c=c(F,x^{ad})$.
\end{proof}
\begin{proof}[Proof of~\eqref{Uprimisproduct}]
Let~$\widetilde{M}$ be the inverse image of~$M^{ab}(\Q)$ by~$M(\A_f)\to M^{ab}(\A_f)$. 

The inverse image in~$U'_x$ of~$\Gamma'_x\cap M^{ad}(\A_f)$ is
\(
\wt{M}/Z(M)(\Q)
\)
and the inverse image of~$\Gamma'_x\cap \frac{M^{ab}(\A_f)}{M^{ab}(\Q)}$ is
\(
Z(M)(\A_f)/Z(M)(\Q)
\). Thus~\eqref{Goursat Gamma} implies
\[
[U'_x:  (U'_x\cap \wt{M}/Z(M)(\Q))\cdot(U'_x\cap Z(M)(\A_f)/Z(M)(\Q))].
\]

By Theorem~\ref{thm:Borel}, we have
\[
[\widetilde{M}:M^{der}(\A_f)\cdot Z(M)(\Q)]
=
[ab(M(\A_f))\cap M^{ab}(\Q):ab(Z(M)(\Q))]\leq g
\]
We deduce~\ref{Uprimisproduct}.
\end{proof}
%

\subsection{Proof of Theorem~\ref{thm:complement}} \phantom{}
\subsubsection{} Let~$Z(p)=Y_{x^{ad},F'}$ and~$p_0$ be as in~\S\ref{sec:Proof of Theorem Uestproduit}
Recall that~$Z(p)$ depends only on~$x^{ad}$ and~$F$ and that~$Z(p)=Y_x(p)$ for~$p\geq p_0$ and that we have
\[
\forall p, Z(p)\leq Y_x(p).
\]
Moreover~$[Y_x(p):Z(p)]\leq m(p):=[Y_{x^{ad},F}(p):Z(p)]<+\infty$. We have~$m:=\prod_p m(p)<+\infty$ and
\[
\left[\prod Y_x(p):\prod Z(p)\right]\leq m.
\]

Let~$Z(p)^\dagger\leq Z(p)\leq M^{ad}(\Q_p)$ be the subgroup generated
by topologically~$p$-nilpotent elements. Then~$[Z(p):Z(p)^{\dagger}]<+\infty$ for every~$p$ because~$Z(p)\leq M^{ad}(\Q_p)$
is compact and~$GL(N,\Z_p)$ has an open prop-$p$-subgroup. For~$p\geq p_0$ the image of~$Z(p)$ in~$GL(N,\F_p)$ is~$V(p)^\dagger$. This implies~$Z(p)=Z(p)^{\dagger}$ for~$p\geq p_0$. This implies
\[
b:=\left[\prod Z(p):\prod Z(p)^{\dagger}\right]<+\infty.
\]

As a consequence we may replace~$Y_{x^{ad},F}$ by~$\prod Z(p)^{\dagger}$ in Theorem~\ref{thm:Uestproduit}.
%
%
%

\subsubsection{} 
Let~$V_x$ be the inverse image of~$\prod Z(p)^\dagger$ in~$U'_x\cap M^{der}(\A_f)$. 
By Theorem~\ref{thm:Uestproduit}, we have
\[
ad_M(U'_x)=ad_M(U'_x\cap M^{der}(\A_f))
\]
For every element~$z\in Z(p)^\dagger$, there exists~$u\in U'_x\cap M^{der}(\A_f)$ such that~$ad_M(u)=z$. Moreover~$\wt{z}:=u^f$ only depend on~$z$. If~$z$ is topologically~$p$-nilpotent, then~$z=ad_M(u)$ for some topologically~$p$-nilpotent~$u\in U'_x\cap M^{der}(\A_f)$ (\cite[Lemma 5.16]{APZ2}) and thus~$\wt{z}$ is topologically~$p$-nilpotent. Let~$V_p:=\wt{Z(p)^{\dagger}}\leq U'_{x}\cap  M^{der}(\Q_p)$
be the group generated by the~$\wt{z}$, as~$z$ ranges through the~$p$-nilpotent elements of~$Z(p)^\dagger$. Then
\[
\forall u\in U_x, u^f\in \prod_p V_p,
\]
and~$V_p$ is generated by topologically~$p$-nilpotent elements 
and~$V:=\prod V_p$ only depends on~$x^{ad},F$.

This~$V$ has all the required properties. This proves Theorem~\ref{thm:complement}.

\section{Adelic bound for Tori in a geometric conjugacy class}\label{sec:adelic tori}

\subsection{Bounds on~$p$-adic orbits}

The following is a generalisation of~\cite[Prop. 4.3.9]{EdYa} in which one allows tori with bad reduction. We prove it in~\S\ref{sec:Th:EYext}. 
\begin{theorem}\label{Th:EYext}
Let~$T\leq GL(N)_{\Q}$ be a torus. There exists~$c=c(T)$ such that for every prime~$p$ and every~$T'\leq GL(N)_{\Q_p}$ which is~$GL(N,\ol{\Q_p})$-conjugated to~$T$, 
\[
[T'_{max}:T'_{max}\cap GL(N,\Z_p)]\neq1
\Rightarrow
[T'_{max}:T'_{max}\cap GL(N,\Z_p)]\geq p/c.
\]
where~$T'_{max}$ denotes the maximal compact subgroup of~$T'(\Q_p)$.
\end{theorem}
Note that~$c(T)$ depends only on the~$GL(n,\ol{\Q})$-conjugacy class of~$T$.

\subsubsection{Facts on subtori of~$GL(N)_\Q$}
Let~$D\leq GL(N)$ and let~$W:=N_{GL(N)}(D)/D\simeq \mathfrak{S}_N$ be the Weyl group. 
Recall that every~$GL(N,\ol{\Q})$-conjugacy class of tori~$T\leq GL(N)_{\ol{\Q}}$ contains a torus~$T\leq D$ and that two tori~$T,T'\leq D$ are~$GL(N,\ol{\Q})$-conjugated if and only if they are~$W$-conjugated.
Let~$D'\leq GL(N)_{\ol{\Q}}$ be a maximal torus and~$E'$ be the~$\ol{\Q}$-vector space of matrices generated by~$D'(\ol{\Q})$.  Then the set~$Hom(E',\ol{\Q})$ of $\ol{\Q}$-algebra morphisms induces a basis of~$X(D'):=\Hom(D',GL(1)_{\ol{\Q}})$, and the set~$Hom(\ol{\Q},E)$ of $\ol{\Q}$-algebra morphisms induces a basis of~$Y(D):=\Hom(GL(1)_{\ol{\Q}},D')$. 
A torus~$T'\leq D'$ is uniquely determined by the corresponding sublattice~$Y(T')\leq Y(D')$, which is a primitive sublattice.

Assume that~$D'$ is defined over~$\Q$. Then the action of~$Gal(\ol{\Q}/\Q)$ on~$Y(D')$ is induced by the permutation action on~$\Hom(\ol{\Q},E)$. The torus~$T\leq D'$
is defined over~$\Q$ if and only if~$Y(T')\leq Y(D')$ is stable by~$Gal(\ol{\Q}/\Q)$. Consider~$g\in GL(N,\ol{\Q})$ suc that~$gD'(\ol{\Q})g^{-1}=D(\ol{\Q})$ and the induced
isormorphism~$\phi_g:Y(D')\to Y(D)$. Then~$\phi_g$ conjugates the permutation action~$Gal(\ol{\Q}/\Q)$ in the permutation group~$\mathfrak{S}(Hom(\ol{\Q},E))$ to an
morphism~$\rho_{D,g}:Gal(\ol{\Q}/\Q)\to \mathfrak{S}_N$.

Let~$p$ be a prime and~$I_p\leq Gal(\ol{\Q}/\Q)$ be the inertia subgroup. Then~$U_{g,D,p}:=\rho_{g,D}(I_p)$ is a subgroup of~$\mathfrak{S}_n$. 
To every subgroup~$U\leq \mathfrak{S}_N$ we associate~$L_U\leq \Z^N$ the primitive sublattice of~$U$-invariant elements. 
Let~$\Lambda=\{L_U\leq S_N\}$ the set of such lattices. The group~$W=S_N$ is finite. This leads to the following observation. 
Let~$g_0\in GL(N,\ol{\Q})$ be such that~$g_0T'g_0^{-1}\leq D$.
Then the set
\begin{multline}
\{Y(gT'g^{-1})\cap L_U|g\in GL(N,\ol{\Q}),gT'g^{-1}\leq D, U\leq S_N\}\\=\{Y(sg_0T'g_0^{-1}s^{-1})\cap L_U|s\in W, U\leq S_N\}
\end{multline}
is a finite set which depends only on~$Y(gT'g^{-1})$.

\subsubsection{Facts on subtori of~$GL(N)_{\Q_p}$}\label{sec:factsQp}
Let~$Y(T)^{I_p}$ denote the sublattices of elements fixed by~$I_p$. 
Then~$Y(T)^{I_p}=Y(T^{nr,p}_{p})$ where~$T^{nr,p}_{\Q^{nr}_p}\leq T_{\Q^{nr}_p}$ is the maximal~$\Q^{nr}_p$-subtorus  which is split. 
By the observations above the~$GL(N,\ol{\Q})$-conjugacy classes of~$T^{nr,p}$, as~$p$ varies, belong to a finite set of conjugacy classes depending
only on the conjugacy class of~$T$.

\begin{lemma}\label{lem:nrfactor}
Let~$t\in T(\Q_p)$ be a torsion element (of order prime to~$p$). Then~$t\in T^{nr,p}(\Q_p)$.
\end{lemma}
\begin{proof}Let~$T'\leq GL(N,\Q)$ be a maximal torus containing~$T$ and~$A$ be the algebra $\Q$-algebra generated by~$T'$. We have then~$T'=GL(1,A)$. 

We prove the lemma for~$T'=T$.
Factoring~$A\tens\Q_p$ as a product of fields we are reduced to the case~$T'(\Q_p)=K^\times$ for a field extension~$K/\Q_p$. 
Let~$O_K$ be the ring of integers and~$\mathfrak{m}$ its maximal ideal and~$\kappa=O_K/\mathfrak{m}$. Then~$O_K^\times / (1+\mathfrak{m})\simeq \kappa^\times$ is finite of order prime to~$p$, and~$1+\mathfrak{m}$ is a pro-$p$-group witout torsion. We can thus factor~$O_K^\times\simeq \Gamma_K\times (1+\mathfrak{m})$, where~$\Gamma_K\simeq \kappa^{\times}$ is the torsion subgroup
of~$O_K^{\times}$. Observe that~$\kappa$ is also the residue field of the maximal unramified extension~$K^{nr}\subseteq K$. We deduce that~$\Gamma_{K^{nr}}\leq \Gamma_K$ satisfy~$\#\Gamma_{K^{nr}}
=\#\kappa^{\times}=\# \Gamma_K$. We have thus~$\Gamma_K=\Gamma_{K^{nr}}\leq T'^{nr,p}(\Q_p)$. 

We deduce the general case. Let~$I\leq Gal(\ol{\Q_p}/\Q_p)$ be the inertia subgroup, acting onr~$Y(T)\leq Y(T')$. Then~$Y(T^{nr,p})\leq Y(T)$, resp.~$Y(T'^{nr,p})\leq Y(T')$ is the subgtoup of~$I$-invariants of~$Y(T)\tens\Q$ belonging to~$Y(T)$ resp.~$Y(T')$. We deduce~$Y(T^{nr,p})=Y(T)\cap Y(T'^{nr,p})$. It follows that~$T'^{nr,p}\cap T$ is a torus and~$T'^{nr,p}\cap T=T^{nr,p}$. The torsion subgroup of~$T(\Q_p)$ is thus contained in~$T'^{nr,p}(\Q_p)\cap T(\Q_p)=T^{nr,p}(\Q_p)$.
\end{proof}
\begin{corollary}Let~$T(\Q_p)^{\dagger}:=\{t\in T(\Q_p)|\lim t^{p^k}=1\}$. We have
\[
T_{max}=T^{nr,p}_{max}\cdot T(\Q_p)^{\dagger}.
\]
\end{corollary}\label{cor:nrfactor}
\begin{proof} Let~$t\in T(\Q_p)_{max}$. Then~$t$ topologically generates a compact subgroup, say~$K$. The compact group~$K$ is commutative, profinite. Moreover~$K^{\dagger}\leq K$ is open. We can thus factor~$K=K^\dagger\cdot K'$ with~$K'$ the subgroup of torsion elements of order prime to~$p$. We have~$K^\dagger\leq T(\Q_p)^{\dagger}$ and, by Lemma~\ref{lem:nrfactor} we have~$K'\leq T^{nr,p}_{max}(\Q_p)$. The corollary follows.
\end{proof}
\subsubsection{Proof of Theorem~\ref{Th:EYext}}\label{sec:Th:EYext}
Assume~$[T'_{max}:T'_{max}\cap GL(N,\Z_p)]$. Then, by Corollary~\ref{cor:nrfactor} we have~
\[
\text{$[T^{nr,p}_{max}:T^{nr,p}_{max}\cap GL(N,\Z_p)]\neq 1$ or~$[T(\Q_p)^{\dagger}:T(\Q_p)^{\dagger}\cap GL(N,\Z_p)]\neq 1$.}
\]
In the first case~\cite[]{EdYa} imply~$[T^{nr,p}_{max}(\Q_p):T^{nr,p}_{max}\cap GL(N,\Z_p)]\geq p/c(T^{nr,p})$, where~$c(T^{nr,p})$ depends only on the $GL(N,\ol{\Q_p})$ conjugacy class of~$c(T^{nr,p})$. By the remarks~\S\ref{sec:factsQp}, the quantity~$c(T^{nr,p})$ is bounded in terms of the $GL(n,\ol{\Q})$-conjugacy class of~$T$. 

In the second case, since~$T(\Q_p)^{\dagger}$ is generated by topologically~$p$-nilpotent elements, we have~$[T(\Q_p)^{\dagger}:T(\Q_p)^{\dagger}\cap GL(N,\Z_p)]\neq 1$.

%
%

\subsection{Bounds on adelic orbits}

\begin{theorem}\label{thm:84}
Let~$T\leq GL(N)_{\Q}$ be a torus and~$h\in\Z_{\geq1}$.  Let~$W\simeq GL(N)/N_{GL(N)}(T)$ be the conjugacy class of~$T$, as a variety over~$\Q$.

For every~$w\in W(\A_f)$ we denote by~$T_w\leq GL(N)_{\A_f}$ the corresponding torus and by~$K_w\leq T_w(\A_f)$ the maximal compact subgroup and
and define~$K_{w,h}:=\{t^h|t\in K_w\}$.

Then, as~$w$ varies in~$W(\A_f)$,
\[
[K_w:K_w\cap GL(N,\wh{\Z})]\approx
[K_{w,h}:K_{w,h}\cap GL(N,\wh{\Z})].
\]
\end{theorem}

We can factor~$K_w=\prod_p K_{w,p}$ and~$K_{w,h}=\prod_p K_{w,h,p}$. Let~$F$ be the set of primes such that
\[
[K_w:K_w\cap GL(N,\wh{\Z})]\neq 1.
\]
For~$p\not\in F$, we have
\[
[K_w:K_w\cap GL(N,\wh{\Z})]=[K_{w,h}:K_{w,h}\cap GL(N,\wh{\Z})]=1
\]
By~\cite[Lem. B.10]{APZ2} we have
\[
[K_{w.p}:K_{w,h,p}]\leq k(h,N)
\]
For~$p\in F$, we have thus
\[
[K_{w,h,p}:K_{w,h.p}\cap GL(N,\Z_p)]\geq \frac{1}{k(h,N)} [K_{w,p}:K_{w,p}\cap GL(N,\Z_p)].
\]
By Theorem~\ref{Th:EYext}. we have
\[
\forall p\in F,[K_{w,p}:K_{w,p}\cap GL(N,\Z_p)]\geq p/c(T).
\]
We deduce
\[
[K_w:K_w\cap GL(N,\wh{\Z})]=
\prod_{p\in F}[K_{w,p}:K_{w,p}\cap GL(N,{\Z_p})]\geq \prod_{p\in F} p/c\approx \prod_{p\in F}p
\]
and
\[
\frac{[K_{w,h}:K_{w,h}\cap GL(N,\wh{\Z})]}{[K_w:K_w\cap GL(N,\wh{\Z})]}
=
\prod_{p\in F}\frac{[K_{w,h,p}:K_{w,h,p}\cap GL(N,{\Z_p})]}{[K_{w,p}:K_{w,p}\cap GL(N,\Z_p)]}
\leq \prod_{p\in F} h^{\dim(T)}.
\]
We deduce
\[
\frac{[K_{w,h}:K_{w,h}\cap GL(N,\wh{\Z})]}{[K_w:K_w\cap GL(N,\wh{\Z})]}\leq (h^{dim})^{\#F}\prec \prod_{p\in F}p\preccurlyeq 
[K_w:K_w\cap GL(N,\wh{\Z})].
\]
This implies
\[
[K_w:K_w\cap GL(N,\wh{\Z})]\approx
[K_{w,h}:K_{w,h}\cap GL(N,\wh{\Z})].
\]

\section{An auxiliary torus}\label{sec:Auxiliary torus}

\subsection{Brauer-Siegel theorem for Tori}\label{sec:Brauer}
Let~$S$ be a~$\Q$-algebraic torus over~$\Q$ such that
\begin{equation}\label{CompactCM}
\text{$S(\R)/S^{spl}(\R)$ is compact,}
\end{equation}
 where~$S^{spl}\leq S$ is
the maximal~$\Q$-split subtorus. This property is invariant by isogeny and holds true for every subtorus~$S\leq T$ such that there exists
a Shimura datum~$(T,x)$.

The (unramified) class group and class number of~$S$ are
\[
cl(S):=
S(\Q)\sous S(\A_f)/ S_{max}
\text{ and }h(S):=\# cl(S)
\]
where~$S_{max}\leq S(\A_f)$ is the maximal compact subgroup.  We denote by~$L(S)$ the splitting field of~$S$ and~$d(S)$ the absolute value
of the discriminant of~$L(S)$. 

\begin{theorem}[{\cite[Th.2.3]{UY4}}] As~$S$ varies among a set of tori satisfying~\eqref{CompactCM} and such that~$\dim(S)$ is bounded, we have
\[
h(S)\approx d(S).
\]
\end{theorem}

\subsection{Lower bounds for the sizes of Galois orbits of special points}
Let~$(T,\{x\})$ be a special Shimura datum and~$E(T,\{x\})$ it reflex field and 
\[
(E(T,\{x\})\tens \A_f)^\times\to T(\A_f)
\]
its reflex norm. Consider the image
\[
\Gamma_{(T,\{x\})}\leq cl(T)
\]
of~$(E\tens \A_f)^\times$ in~$cl(T)$. Let~$S$ be a subtorus(subgroup)
and~$\pi:cl(T)\to cl(S)$ the quotient morphism. Denote by~$(T/S,\{\pi(x)\})$ the quotient Shimura datum.

By Th.~\ref{thm:Borel}, we have
\[
\#\ker\pi \leq \#Sha(S)\cdot \#cl(S).
\]
It follows
\begin{equation}\label{eq:descend AO bound}
\#\pi(\Gamma_{(T/S,\{\pi(x)\})})=\frac{\#\Gamma_{(T,\{x\})}}{\#\Gamma_{(T,\{x\})}\cap \ker(\pi)}\geq 
\frac{\#\Gamma_{(T,\{x\})}}{\#Sha(S)\cdot \#cl(S)}.
\end{equation}
Denote by~$L(S),L(T)$ and~$L(T/S)$ the spltiing field of~$S,T$ and~$T/S$ and denote by~$d(S),d(T)$ and~$d(T/S)$ the absolute value of their discriminant.
We have~$L(T)=L(S)\cdot L(T/S)$ and we deduce~$d(S)\cdot d(T/S)\approx d(T)$ as~$\dim(L(T))$ is bounded.



\begin{proposition} Assume~$(G,X)$ to be of abelian type. As~$(T,x)$ ranges through all special subdata of~$(G,X)$, we have
\[
\#\Gamma_{T,x}\succcurlyeq d(T).
\]
\end{proposition}
\begin{proof}

In the case~$(G,X)=(GSp(2g),H_{g})$ this is~\cite[(1) p.\,386]{TsiAG}. 

The case where~$(G,X)$ is a subdatum of~$(GSp(2g),H_{g})$
follows immediately.

Let~$(G',X')$ be a subdatum of~$(GSp(2g),H_{g})$. Then every CM subdatum~$(T,x)$ of~$(G'^{ad},X'^{ad})$
is of the form~$(T'/Z(G),x'^{ad})$ for a CM Shimura subdatum~$(T',x')$  of~$(G',X')$. Let~$S:=Z(G)$. By~\eqref{eq:descend AO bound} we have
\begin{equation}\label{eq:truc}
\#\Gamma_{T,x}\geq \#\Gamma_{T',x'}\cdot f(S)\succcurlyeq f(S)\cdot d(T')\approx f(S)d(S)d(T).
\end{equation}
with~$f(S)>0$ and~$d(S)$ depending only on~$G'$.

The Shimura datum~$(G,X)$ is of abelian type of and only if there is a subdatum~$(G',X')$ of~$(GSp(2g), H_{g})$ such that~$(G^{ad}, X^{ad})=(G'^{ad}, X'^{ad})$.
As~$(T,x)$ ranges through CM subdata of~$(G,X)$, we can apply~\eqref{eq:truc} to~$(T/Z(G),x^{ad})$. We deduce
\[
\#\Gamma_{T,x}\geq \#\Gamma_{T/Z(G),x^{ad}}\succcurlyeq d(T/Z(G))\approx d(T)/d(Z(G)).
\]
with~$d(Z(G))$ depending on on~$G$. 
\end{proof}

\subsection{An auxiliary Torus} Let~$(M,X_M)\leq (G,X)$ be a Shimura subdatum and~$(T,\tau)$ be a CM Shimura subdatum and~$(S',\tau':=ad_{M}(\tau))\leq (M^{ad},X^{d}_M)$ be the image CM Shimura subdatum.
Let~$S:=T\cap M^{der}$, and let~$x\in (M^{ad},X^{d}_M)$ be a Hodge generic point. Let
\[
\Gamma_{\pi_0,M,x,F}=...
\]
We have then
\[
\Gamma_{\pi_0,M,x,F}\cdot[F\cap \ol{\Q}:E(M^{ab},x^{ab})]\geq\Gamma_{(M^{ab},x^{ab})}\succcurlyeq f(S)\cdot d(M^{ab})=f(S)\cdot Z(M)
\]
Assume~$(G, X)$ is of abelian type or assume GRH. Then, as~$x$ varies in the hybrid Hecke orbit, we have
\begin{equation}\label{eq:dbound1}
\Gamma_{\pi_0,M,x,F}\succcurlyeq d(Z(M)).
\end{equation}
Moreover
\begin{equation}\label{eq:dbound2}
[F([x,1]):F]\geq \Gamma_{\pi_0,M,x,F}.
\end{equation}

\section{Bound for the geometric part of Galois orbits}\label{sec:geometric bounds}
We prove the following.

\begin{theorem}\label{th:sec10}
We consider the setting of Conjecture~\ref{conj:Bounds bis} and~\S\ref{sec:dicho}. Assume that~$S=Sh_K(G,X)$ with~$(G,X)$ of abelian type.
Then, as~$x'$ varies in the geometric hybrid orbit~$\Sigma^g(x)$,
\begin{equation}\label{eq:bound geo}
H_f(w(x')^{cent})\cdot [F([x',1])):F]\succcurlyeq H_f(w(x')^{der}).
\end{equation}
\end{theorem}
In proving~\eqref{eq:bound geo}, we can pass to the adjoint datum~$(G^{ad},X^{ad})$. In the rest of \S\ref{sec:geometric bounds}, we assume that
\begin{equation}\label{eq:G is adj}
G=G^{ad}.
\end{equation}

Let~$V=\prod_p V_p$ as in Theorem~\ref{thm:complement}. 
We first reduce the estimate~\eqref{eq:bound geo} to an estimate on adelic obits.
\begin{lemma}\label{lem:reduction}
 There exists~$c(G,X)\in\R_{>0}$ such that, for every~$x'$ in the geometric hybrid orbit of~$x$,
\begin{equation}\label{69cmieux}
\# Gal(\ol{F}/F\cdot E(M_{x'},X_{x'}))\cdot [x',1] \geq c(G,X)\cdot [V:V\cap K].
\end{equation}
\end{lemma}
\begin{proof}[Proof of Lemma~\ref{lem:reduction}] Observe that $M_{x'}^{der}(\A_f)\to M_{x'}(\A_f)/Z(M_{x'})(\Q)$ is a map of degree at most~$\#Z(M^{der}_{x'})(\Q)$, which can be bounded in terms of~$G$.
We can thus argue as in~\eqref{eq:gal vs adel}--\eqref{eq:gal vs adel2}.
\end{proof}

We will apply~\cite[Th.~6.1]{APZ2}. We need to introduce vectors in representations of~$G$.

We observe that~$\exp(2p\mathfrak{gl}(N,\Z_p))$ is an open subgroup of~$GL(N,\Q_p)$, and that, for~$p\gg0$, every topologically~$p$-nilpotent element~$u\in GL(N,\Q_p)$
is of the form~$u=\exp(X)$ with~$X=\log(u)$. After replacing~$V$ by an open subgroup, there are finite subsets~$F_p\subseteq V_p$ such that for every~$p$ the subgroup generated by~$F_p\subseteq V_p$ is dense in~$V_p$ 
and such that every~$f\in F_p$ we have~$f=\exp(\log(f))$. This implies that~$f$ is topologically~$p$-nilpotent.

We consider a finite subset~$D\subseteq M_x^{der}(\Q)$, resp.~$E\subseteq Z(M_x)(\Q)$, which generates a Zariski dense subgroup of~$ M_x^{der}$, resp.~$Z(M_x)$.


The algebraic group~$G$ acts on~$\mathfrak{gl}(N)_{\Q}$ via the adjoint representation. We deduce a representation in~$\mathfrak{gl}(N)_{\Q}^{F_p}$, in~$\mathfrak{gl}(N)_{\Q}^{D}$ and in~$\mathfrak{gl}(N)_{\Q}^{E}$.
The stabiliser of~$(f)_{f\in F_p}$, of~$(d)_{d\in D}$ and~$(e)_{e\in E}$ is respectively
\[
Z_G(V_p)(\Q_p)\text{ and }Z_G(M_x^{der})(\Q_p)\text{ and }Z_G(Z(M_x))(\Q_p).
\] 
The stabilisers of~
\[
v:=((f)_{f\in F_p},(e)_{e\in E})\text{ and~}v':=((d)_{d\in D},(e)_{e\in E})
\] are thus
\[
Z_G(V_p)(\Q_p)\cap Z_G(Z(M_x))(\Q_p)\text{ and } Z_G(M_x^{der})(\Q_p)\cap Z_G(Z(M_x))(\Q_p)
\]
respectively.

For~$u\in GL(N,\Q_p)^{F_p}$, let us define 
\begin{equation}\label{eq:ht truc}
H_p(u):=\max\{p^k|p^k \cdot u\in \End(\Z_p^N)^{F_p}\}
\end{equation}
and similarly for~$u\in GL(N,\Q_p)^{D}$ and~$u\in GL(N,\Q_p)^{E}$.

As~$x'\in \Sigma^g(x)$, we have a distinguished isomorphism~$\phi_{x'}:M^{der}_x\simeq M^{der}_{x'}$. Let~$V_{p,x'}:=\phi_{x'}(V_p)\leq M^{der}_{x'}(\Q_p)\leq GL(N,\Q_p)$ be the image of~$V_p$ 
and~$u'=u(x'):=(\phi_{x'}(f))_{f\in F_p}$. We claim the following:
\begin{equation}\label{69cmieuxbise}
[V_{p,x'}:V_{p,x'}\cap K]\succcurlyeq H_p(u').
\end{equation}
\begin{proof}For~$f\in F_p$, let~$\phi_{x'}(f)^{\Z_p}\leq V_{p,x'}$ denote the closed subgroup generated by~$\phi_{x'}(f)$. We have
\[
[V_{p,x'}:V_{p,x'}\cap K]\geq \max_{f\in F_p} [\phi_{x'}(f)^{\Z_p}:\phi_{x'}(f)^{\Z_p}\cap K].
\]
On the other hand~$H_p(u')\approx \max_{f\in F_p} H_p(\phi_{x'}(f))$.
It is enough to prove that~$\forall f\in F_p,[\phi_{x'}(f)^{\Z_p}:\phi_{x'}(f)^{\Z_p}\cap K]\succcurlyeq H_p(\phi_{x'}(f))$.

Recall that~$f=\exp(\log(f))$. It follows that there exists~$P_f\in\Q_p[T]$ such that~$f=P_f(\log(f))$. We deduce that~$\phi_{x'}(f)=P_f(\log(\phi_{x'}(f)))$
and 
\begin{equation}\label{eq:domination exp matrice}
H_p(\phi_{x'}(f))\preccurlyeq H_p(\log(\phi_{x'}(f))).
\end{equation}
Arguing as in~\cite[Prop. A1 (68)]{APZ1}, there is~$p(N)$ such that
\begin{equation}\label{eq:domination exp matrice uniforme}
\forall p\geq p(N),H_p(\phi_{x'}(f))\leq H_p(\log(\phi_{x'}(f)))^{N-1}.
\end{equation}
By~\cite[Th. A3]{APZ1}, we have
\[
[\phi_{x'}(f)^{\Z_p}:\phi_{x'}(f)^{\Z_p}\cap K]\geq \frac{1}{N}\cdot H_p(\log(\phi_{x'}(f)))\succcurlyeq
H_p(\phi_{x'}(f)).\qedhere
\]
\end{proof}


 Note that~$V\leq M^{der}(\A_f)\leq GL(N,\A_f)$ is a compact subgroup. Therefore~$(f)_{f\in F_p}\in GL(N,\Z_p)^{F_p}$ for~$p\gg0$. As~$(d)_{d\in D}$ and~$(e)_{e\in E}$ are defined over~$\Q$, we have~$(d)_{d\in D}\in GL(N,\Z_p)^{D}$ and~$(e)_{e\in E}\in GL(N,\Z_p)^{E}$ for~$p\gg0$.
For sufficiently large~$p$, we can define
\[
\ol{v}\in \End(\F_p^N)^D\times \End(\F_p^N)^E\qquad \ol{v'}\in \End(\F_p^N)^{F_p}\times \End(\F_p^N)^E.
\]

In order to use~\cite[Th. 6.1]{APZ2}, we need to check that the assumptions are satisfied.

The datum~$(G,X)$ is of abelian type and~$G$ is adjoint by \eqref{eq:G is adj}. By~\cite[\S 4.3\text{ and }Cor. 4.11]{APZ2},
the point~$x\in X$ satisfies~\cite[Déf. 2.1, 2.3]{APZ2}. By~\cite[Cor. B.11]{APZ2} and Th.~\ref{thm:puissancecentre} and \eqref{thm:goursat:eq:exposant} and~\eqref{eq:thm:cor CU},
the group~$W:=V\cdot \{z^h|z\in Z(M_{x})_{max}\}$  satisfies~\cite[Def. 2.1]{APZ2}.

By~(1a) of~\cite[Def. 2.1]{APZ2}, we have
\[
Z_{G(\Q_p)}(V_p\cdot Z(M_x))=Z_G(M_x)(\Q_p)= Z_G(M^{der}_x)\cap Z_G(Z(M_x)).
\]
This implies
\begin{equation}\label{eq:assumption 1}
Stab_{G_{\Q_p}}(v)=Stab_{G_{\Q_p}}(v').
\end{equation}
By~(1b) of~\cite[Def. 2.1]{APZ2}, the action of~$V\cdot \{z^h|z\in Z(M_{x})_{max}\}$ on~$\Q_p^N$ is semisimple.
This implies that~$\ol{W }^{Zar}\leq GL(N)_{\Q_p}$ is reductive. By~\cite{Richardson}, we deduce that the orbit
\[
G\cdot v\subseteq G^{F_p\sqcup D}
\]
is Zariski closed. As~$G$ is adjoint, we have~$G\leq SL(N)$. We deduce that inclusions of subvarieties
\[
G\cdot v\subseteq G^{F_p\sqcup E}\subseteq SL(N)^{F_p\sqcup E}_{\Q_p}\subseteq \End(\Q_p^N)^{F_p\sqcup E}
\]
are Zariski closed. Likewise,~$M_x$ being reductive,~\cite{Richardson} implies that
\[
G\cdot v'\subseteq G^{D\sqcup E}\subseteq SL(N)^{D\sqcup E}_{\Q_p}\subseteq \End(\Q_p^N)^{D\sqcup E}
\]
is Zariski closed.

We note that the inclusion~$Z(M_x),M_x,G\leq GL(N)$ induces models~$Z(M_x)_{\Z_p},{M_x}_{\Z_p},G_{\Z_p}\leq GL(N)_{\Z_p}$, which, for~$p\gg0$ are hyperspecial. The fibers~$Z(M_x)_{\F_p},{M_x}_{\F_p},G_{\F_p}\leq GL(N)_{\F_p}$ are connected smooth reductive subgroups for~$p\gg0$.
For~$p\gg0$, the image~$\ol{W}$ of~$W$ by~$GL(N,\Z_p)\to GL(N,\F_p)$ is generated by the images~$\ol{f}$ of the~$f\in F_p$ and by~$\{z^h|z\in Z(M_{x})(\F_p)\}$. 

We have~$Stab_{G_{\F_p}}(\ol{v})=Z_{G_{\F_p}}(\{\ol{f}|f\in F_p\})\cap Z_{G_{\F_p}}(\{\ol{e}|e\in E\})$. We claim that
\[
\forall p\gg0,
Z_{G_{\F_p}}(\{\ol{e}|e\in E\})=Z_{G_{\F_p}}(Z(M_{x})_{\F_p})=Z_{G_{\F_p}}(\{z^h|z\in Z(M_{x})(\F_p)\}).
\]
\begin{proof}Consider~$z_1,\ldots,z_k\in Z(M_x)(\Q)$ such that~${z_1}^h\Q+\ldots+{z_k}^h\Q\subseteq End(\Q^N)$ is the~$\Q$-algebra
generated by~$Z(M_x)$. We have thus~$Z_G(Z(M_x))=Z_G(A)=Z_G(\Lambda)$ with~$\Lambda:={z_1}^h\Z+\ldots+{z_k}^h\Z$. For~$p\gg0$,
we can reduce the~$z_i$ and obtain~${\ol{z_1}}^h,\ldots,{\ol{z_k}}^h\in \{z^h|z\in Z(M_{x})(\F_p)\}$. 

Arguing as in~\cite[proof of the claim, p.\,16]{APZ2} , for~$p\gg0$, we have~$(Z_G(A))_{\F_p}=Z_{G_{\F_p}}(\{\ol{z_1}^h;\ldots;\ol{z_k}^h\})$.
For~$p\gg0$, the group scheme~$Z(M_x)_{\Z_p}$ is smooth over~$\Z_p$. Let~$B$ the~$\ol{\F_p}$-module generated
by-$Z(M_x)(\ol{\F_p})$, and let~$y_1,\ldots,y_l\in Z(M_x)(\ol{\F_p})$ be a basis of~$B$. By smoothness there are 
lifts~$\wt{y_1},\ldots,\wt{y_l}\in Z(M_x)(\ol{\Z_p})$ of the~$y_i$. Let~$\wt{B}$ be the~$\ol{\Q_p}$-vector space generated
by the~$\wt{y_i}$. We have~$\wt{B}\subseteq A\tens\ol{\Q_p}$ and thus~$l=\dim(B)\geq \dim(\wt{B})\leq \dim(A)$. But~$\ol{\Lambda}:=\Lambda\pmod{p}\subseteq B$
and~$\dim(\Lambda\pmod{p}\subseteq B)=\dim(A)$ for~$p\gg0$. We deduce~$B=\ol{\Lambda}$ for~$p\gg0$, and thus
\[
Z_{G_{\F_p}}(Z(M_x)_{\F_p})=Z_{G_{\F_p}}(B)=Z_{G_{\F_p}}(\ol{\Lambda})\supseteq 
Z_{G_{\F_p}}(\{\ol{z_1}^h;\ldots;\ol{z_k}^h\})=(Z_G(A))_{\F_p}.
\]
Note that~$\{\ol{z_1}^h;\ldots;\ol{z_k}^h\}\subseteq\{z^h|z\in Z(M_{x})(\F_p)\}\subseteq Z(M_x)_{\F_p}$ and thus
\begin{multline}
Z_{G_{\F_p}}(Z(M_x)_{\F_p})\subseteq Z_{G_{\F_p}}(\{z^h|z\in Z(M_{x})(\F_p)\})\\
\subseteq Z_{G_{\F_p}}\{\ol{z_1}^h;\ldots;\ol{z_k}^h\})\subseteq Z_{G_{\F_p}}(Z(M_x)_{\F_p}).\qedhere
\end{multline}
\end{proof}
We deduce
\begin{multline}
\forall p\gg0,
Stab_{G_{\F_p}}(\ol{v})=Z_{G_{\F_p}}(\{\ol{f}|f\in F_p\})\cap Z_{G_{\F_p}}(\{z^h|z\in Z(M_{x})(\F_p)\})\\
=Z_{G_{\F_p}}(\{\ol{f}|f\in F_p\})\cap Z_{G_{\F_p}}(Z(M_{x})_{\F_p})=Z_{G_{\F_p}}(\ol{W}).
\end{multline}
By~(2a) of~\cite[Def. 2.1]{APZ2}, we obtain
\[
\forall p\gg0, Stab_{G_{\F_p}}(\ol{v})=Z_{G_{\F_p}}(M_{\F_p})=Stab_{G_{\F_p}}(\ol{v'}).
\]

By~(2b) of~\cite[Def. 2.1]{APZ2}, the action of~$\ol{W}$ on~$\F_p^N$ is semisimple for~$p\gg0$.
By~\cite[Th.\,5.4]{SerreCR}, this implies that~$\ol{W}$ is strongly reductive in~$G_{\F_p}$ for~$p\gg0$.
By~\cite[Th. 3.7]{SerreCR} this implies that
\[
G_{\F_p}\cdot \ol{v}\subseteq G_{\F_p}^{F_p\sqcup E}\subseteq SL(N)_{\F_p}^{F_p\sqcup E}\subseteq \End(\F_p^N)^{F_p\sqcup E}
\]
is Zariski closed. One can show that for~$p\gg0$, the subgroup generated by the image of~$D\sqcup E$ in~$G(\F_p)\leq GL(N,\F_p)$ 
acts semisimply on~${\F_p}^N$. This
this implies that
\[
G_{\F_p}\cdot \ol{v'}\subseteq G_{\F_p}^{D\sqcup E}\subseteq SL(N)_{\F_p}^{D\sqcup E}\subseteq \End(\F_p^N)^{D\sqcup E}
\]
is Zariski closed. 

We can now use~\cite[Th. 6.1]{APZ2}.

We deduce that there exists~$C\in\R_{>0}$, depending only on the representation~$G\to GL(\mathfrak{g})$ such that, for~$p\gg0$,
\begin{multline}\label{eq1}
\frac{1}{C}\cdot \log H_p(g\cdot (f)_{f\in F_p},(e)_{e\in E}))\\
\leq
\log H_p(g\cdot ((f)_{f\in D},(e)_{e\in E}))\\
\leq C\cdot \log H_p(g\cdot (f)_{f\in F_p},(e)_{e\in E})) 
\end{multline}
as functions of~$g\in G(\C_p)$. Note that, for every~$p$, the functions~$g\mapsto H_p(g\cdot (f)_{f\in F_p},(e)_{e\in E}))$
and~$g\mapsto H_p(g\cdot ((f)_{f\in D},(e)_{e\in E}))$ are induced by local height functions on the same variety~$G/Stab(v)\simeq G/Stab(v')$.
We thus have
\begin{equation}\label{eq2}
\forall p, H_p(g\cdot ((f)_{f\in F_p},(e)_{e\in E})))\approx H_p(g\cdot ((f)_{f\in D},(e)_{e\in E}))).
\end{equation}

We deduce from~\eqref{eq1} and~\eqref{eq2} that
\[
\prod_p H_p(g\cdot ((f)_{f\in F_p},(e)_{e\in E})))\approx
\prod_p H_p(g\cdot ((f)_{f\in D},(e)_{e\in E}))).
\]

By~\eqref{69cmieux}, we have
\begin{multline}
\# Gal(\ol{F}/F\cdot E(M,x))\cdot [x':1] \geq c(G,X)\cdot \#V/V\cap K \\
=\prod_p[V_p:V_p\cap K]
\end{multline}
Note that~\eqref{69cmieuxbise} is uniform as~$p$ varies, thanks to~\eqref{eq:domination exp matrice uniforme}.
We can multiply over primes and we get
\[
\prod_p[V_p:V_p\cap K]
\succcurlyeq \prod_p H_p(g\cdot ((f)_{f\in F_p})).
\]
We deduce~$\# Gal(\ol{F}/F\cdot E(M,x))\cdot [x':1]\succcurlyeq \prod_p H_p(g\cdot ((f)_{f\in F_p}))$. Multiplying by~$H_f(g\cdot ((e)_{e\in E}))$
we obtain 
\begin{multline}
H_f(g\cdot ((e)_{e\in E}))\cdot 
\# Gal(\ol{F}/F\cdot E(M,x))\cdot [x':1]\\
\succcurlyeq H_f(g\cdot ((e)_{e\in E}))\cdot \prod_p H_p(g\cdot ((f)_{f\in F_p}))\\
\approx \prod_p H_p(g\cdot ((f)_{f\in F_p},(e)_{e\in E}))\\
\approx \prod_p H_p(g\cdot ((f)_{f\in D},(e)_{e\in E}))\\
=H_f(g\cdot ((f)_{f\in D},(e)_{e\in E}))
\geq H_f(g\cdot ((f)_{f\in D}).
\end{multline}
We have proved~\eqref{eq:bound geo}
\section{Proof of Theorem~\ref{thm:tech}}

Let~$(G,K)$ be a Shimura datum, let~$K\leq G(\A_f)$ be an open compact subgroup, and denote by~$S=Sh_K(G,X)$ the corresponding Shimura variety. Let us consider~$x\in X$, and denote by~$(M,X_M)$ the smallest Shimura subdatum of~$(G,X)$ such that~$x\in X_M$. We consider the geometric hybrid orbit~$\Sigma^g(x)\subseteq X$. We choose an extension of finite type~$F/E(G,X)$ such that~$[x,1]\in S(F)$.

For~$x'\in \Sigma^g(x)$, we denote by~$(M_{x'},X_{x'})$  the smallest Shimura subdatum of~$(G,X)$ such that~$x\in X_{x'}$, and define~$K_{x'}:=K\cap M_{x'}(\A_f)$. The morphism of Shimura varieties
\[
Sh_{K_{x'}}(M_{x'},X_{x'})
\to
Sh_K(G,X)
\]
has finite degree, and this degree has an upper bound~$c(S)$ depending only on~$S$. If~$\wt{s'}\in 
Sh_{K_{x'}}(M_{x'},X_{x'})(\ol{F})$ is an inverse image of~$s':=[x',1]\in 
Sh_K(G,X)(\ol{F})$ we have 
\begin{equation}\label{eq:gal vs adel}
[F(s'):F]\geq [F(\wt{s'}):F]/c(S).
\end{equation}

Consider~$\rho_{x'}$ and~$U_{x',F}$ as in~\eqref{eq:defi rhox} and~\eqref{eq:defi Uxf}. Let~$\ol{K_{x'}}:=\frac{K_{x'}\cdot  Z(M_{x'})(\Q)}{Z(M_{x'})(\Q)}\leq \frac{M_{x'}(\A_f)}{Z(M_{x'})(\Q)}$. We then have
\begin{equation}\label{eq:gal vs adel2}
[F(\wt{s'}):F\cdot E(M_{x'},X_{x'})]=[U_{x',F}:U_{x',F}\cap \ol{K_{x'}}].
\end{equation}
Recall that~$[F\cdot E(M_{x'},X_{x'}):F]\leq [E(M_{x'},X_{x'}):\Q]$ is bounded in terms of~$G$ only.

Let~$V=\prod_p V_p$ be as in~Theorem~\ref{thm:complement} and~$h\in\Z_{\geq1}$ be as in Theorem~\ref{thm:puissancecentre}. 

Let~$Z(M_{x'})_{max}\leq Z(M_{x'})(\A_f)$ be the maximal compact subgroup.
The fibres of
\[
Z(M_{x'})_{max}\cdot K_{x'}/K_{x'}\to M_{x'}(\A_f)/(K_{x'}\cdot Z(M_{x'}(\Q))) 
\]
are of finite degree, of cardinality bounded above by~$c'(\dim(G))$ depending only on~$\dim(G)$. Let~$\Theta:=\Theta_h(Z(M_{x'})_{max}):=\{z^{h}|z\in Z(M_{x'})_{max}\}$,
and let~$\ol{\Theta}:=\ol{\Theta}_h(Z(M_{x'})_{max}):=\Theta_h(Z(M_{x'})_{max})\cdot Z(M_{x'}(\Q))/\cdot Z(M_{x'}(\Q))$.
Then
\[
\frac{1}{c'(\dim(G))}[\Theta:\Theta\cap K_{x'}]\leq [\ol{\Theta}:\ol{\Theta}\cap \ol{K_{x'}}].
\]
By Theorem~\ref{thm:puissancecentre} we have~$\ol{\Theta}_h(Z(M_{x'})_{max})\leq U_{x',F}$. Thus
\begin{equation}\label{eq:truc}
[F(\wt{s'}):F]\geq  [\ol{\Theta}:\ol{\Theta}\cap \ol{K_{x'}}]\geq \frac{1}{c(\dim(G))} [\Theta:\Theta\cap K_{x'}].
\end{equation}
By Theorem~\ref{thm:84} we have
\begin{equation}\label{eq:truc22}
\frac{1}{c(\dim(G))} [\Theta:\Theta\cap K_{x'}]\approx [Z(M_{x'})_{max}:Z(M_{x'})_{max}\cap K_{x'}].
\end{equation}

Recall that~$(G,X)$ is of abelian type by assumption. Let~$d(Z(M_{x'}))$ be as in~\S\ref{sec:Brauer}.
By~\eqref{eq:dbound1}, \eqref{eq:dbound2}, we have
\begin{equation}\label{eq:truc3}
[F([x',1]):F]\succcurlyeq d(Z(M_{x'})).
\end{equation}

From~\eqref{eq:truc}, \eqref{eq:truc22} and \eqref{eq:truc3} we deduce
\[
[F([x',1]):F]\succcurlyeq d(Z(M_{x'}))\cdot  [Z(M_{x'})_{max}:Z(M_{x'})_{max}\cap K_{x'}].
\]
By Theorem~\ref{thm:AOvsHeight} we have
\[
H^{cent}_f(w(x')):=\delta_{Z(M_{x'}))}\approx
d(Z(M_{x'})\cdot  [Z(M_{x'})_{max}:Z(M_{x'})_{max}\cap K_{x'}].
\]
We have proved
\[
[F([x',1]):F]\succcurlyeq H_f^{cent}(x').
\]

By Theorem~\ref{th:sec10} we have~$H_f^{cent}(x')\cdot [F([x',1]):F]\succcurlyeq H_f^{der}(x')$. By~\eqref{Route 1} this implies~\eqref{eq:42}.

Theorem~\ref{thm:tech} is proven.

\section{The archimedean part of the height}\label{sec:archimedan part}
\emph{We adapt~\cite[\S5 Th.5.16]{APZ1} to the our more general context.}
\begin{theorem} \label{thm:111}
Let~$\mathfrak{S}\subseteq G(\R)^{+}$ be a finite union of Siegel sets (adapted to~$K_\infty$) and~$W_{\Sigma}\subseteq W(\R)$
be the image of~$\mathfrak{S}$. Let~$H:W(\Q)\to \Z_{\geq1}$ be a global height function and~$H=H_{\R}\cdot H_f$ its factorisation in archimedean and finite part.

Then, as functions on~$W_{\Sigma}\cap W(\Q)$, we have
\[
H_\R\preccurlyeq H_f.
\]
\end{theorem}
We follow~\cite[Theorem 5.16]{APZ1} and its proof. We apply the arguments to~$W=G/(Z(G(M^{der})\cdot N_G(Z(M)))$ instead of~$W=G/Z(G(M))$ in~\cite[\S5]{APZ1}.
We define~\cite[(38)]{APZ1} using~$Z(G(M^{der})\cdot N_G(Z(M))(\R)$ instead of~$Z(G(M))(\R)$. The rest of~\cite[\S5.1.1]{APZ1} holds true with~$Z(G(M^{der})\cdot N_G(Z(M))(\R)$ instead of~$Z(G(M))(\R)$.
The proof of~\cite[\S 5.5.2 i) (40) ii) and iii)]{APZ1} does not change, and thus we can deduce Theorem~\ref{thm:111} from~\cite[Prop. 5.9]{APZ1}.

\section{Galois invariance of the height}\label{sec:invariance}
\emph{We adapt~\cite[Prop. 4.3]{APZ1} to hybrid geometric orbits.}

We consider a Shimura variety~$S=S_K(G,X)$ and a geometric hybrid orbit~$\Sigma^g(x)=\pi(W(\Q)^+)\subseteq X$ and a field~$F$ such that
\[
[x,1]\in S(F).
\]
We consider the Galois representation
\[
\rho_{x^{ad}}:Gal(\ol{F}/F)\to M^{ad}(\A_f).
\]
Let~$d:=\dim(M^{ad})$ and~$U:=\rho_{x^{ad}}(Gal(\ol{F}/F))$. We choose a basis~$(b_1,\ldots,b_{d})$ of~$\mathfrak{m}^{ad}_x$ such that~$\mathfrak{m}^{ad}_{x,\Z}:=b_1\Z+\ldots+b_d\Z$ is such that~$\mathfrak{m}^{ad}_{x,\Z}\tens\wh{\Z}$ is stable under~$U$. We choose a basis~$(c_1,\ldots,c_{\dim(G)})$ of~$\mathfrak{g}$ such that~$\mathfrak{g}_{\Z}:=c_1\Z+\ldots+c_{\dim(G)}\Z$ is such that~$\mathfrak{g}_{x,\Z}\tens\wh{\Z}$ is stable under~$K$.
We choose a faithful representation~$\rho_G:G\to GL(N)$ such that~$K\leq GL(N,\wh{\Z})$.

For every~$w\in W(\Q)^{+}$ the element~$\pi(w)$ has a Mumford-Tate group~$M_{\pi(w)}$ and there is a unique morphism~$\phi_w:M_w\to M_x^{ad}$ such that~$\phi\circ \pi(w)=x^{ad}$. The morphism~$d\phi:\mathfrak{m}_{\pi(w)}\to 
\mathfrak{m}^{ad}_x$ induces an isomorphism~$\psi_{\pi(w)}:\mathfrak{m}^{der}_{\pi(w)}\to \mathfrak{m}^{ad}_x$. We view the inverse matrix as a linear map~$\psi^{-1}_{\pi(w)}\in \Hom(\mathfrak{m}^{ad}_x,\mathfrak{g})$. We can define
\[
H_f(g^{-1}\psi^{-1}_{\pi(w)}g)=\{\min n\in\Z_{\geq1}| n\cdot \psi^{-1}_{\pi(w)}(\mathfrak{m}^{ad}_{x,\Z}\tens\wh{\Z})\leq g\cdot\mathfrak{g}_{\Z}\tens\wh{\Z}\cdot g^{-1}\}
\]
Let~$T_w=\rho_G(Z(M_{\pi(w)})\leq GL(N)$ has an associated canonical tensor~$\eta_{T_{\pi(w)}}\in V_{\dim(T_{\pi(w)})}$ (see~\eqref{eq:canonical embedded tensor}). For~$g\in G(\A_f)$ we can define the finite height
\[
\delta_{g^{-1}T_{\pi(w)}g}:=H_f(g\cdot \eta_{T_{\pi(w)}})=\{\min n\in\Z_{\geq1}|n\cdot \eta\in g\cdot V_{\dim(T_{\pi(w)}),\Z}\tens\wh{\Z}\}.
\]
\begin{proposition}\label{prop:K U M invariance}
The functions
\[
d:
W(\Q)^{+}\times G(\A_f)\to \Z_{\geq1}: (w,g)\mapsto \delta_{g^{-1}T_{\pi(w)}g}
\]
and
\[
h:
W(\Q)^{+}\times G(\A_f)\to \Z_{\geq1}: (w,g)\mapsto H_f(g^{-1}\psi_{\pi(w)}^{-1})
\]
are right~$K$-invariant and left~$G(\Q)$-invariant. For~$w\in W(\Q)^{+}$ and~$g\in G(\A_f)$ and~$m\in M_{\pi(w)}(\A_f)$ such that~$\phi_w(m)\in U$, we have
\[
(d(w,mg),h(w,mg))=(d(w,g),h(w,g)).
\]
\end{proposition}
\begin{proof}The~$K$-invariance follows immediately from~$\rho_G(K)\leq GL(N,\wh{\Z})$, the~$GL(N,\wh{\Z})$-invariance of~$V_{\dim(T_{\pi(w)}),\Z}\tens\wh{\Z}$
and the~$K$-invariance of~$\mathfrak{g}_{\Z}\tens\wh{\Z}$.

The~$G(\Q)$-invariance follows from: 
\[
\forall \gamma\in G(\Q),
(\gamma g)^{-1}T_{\gamma\pi(w)}(\gamma g)=(g^{-1}\gamma^{-1})(\gamma T_{\pi(w)}\gamma^{-1})\cdot(\gamma g)=g^{-1} T_{\pi(w)}g
\]
and
\[
\forall \gamma\in G(\Q),
g^{-1}\psi_{\gamma \pi(w)}^{-1}=g^{-1}\gamma \psi_{\pi(w)}^{-1}=(\gamma^{-1}g)^{-1}\psi_{\pi(w)}^{-1}.
\]

Recall that~$T_{\pi(w)}$ is the centre of~$M_{\pi(w)}$. Therefore, for~$m\in M_{\pi(w)}(\A_f)$ we have~$m^{-1}T_{\pi(w)}m=T_{\pi(w)}$. It follows
\[
\delta_{g^{-1}m^{-1}T_{\pi(w)}mg}=\delta_{g^{-1}T_{\pi(w)}g}.
\]

We observe that~$m^{-1}\psi_{\pi(w)}^{-1}=\psi_{\pi(w)}^{-1}\phi_w(m)$ for~$m\in M_{\pi(w)}(\A_f)$. If~$\phi_w(m)\in U$, we have
\[
g^{-1}m^{-1}\psi_{\pi(w)}^{-1} (\mathfrak{m}^{ad}_x)=g^{-1}\psi_{\pi(w)}^{-1} (\phi_w(m)\cdot \mathfrak{m}^{ad}_x)=g^{-1}\psi_{\pi(w)}^{-1} (\mathfrak{m}^{ad}_x)
\]
and thus
\[
H_f((mg)^{-1}\psi_{\pi(w)}^{-1})=H_f(g^{-1}\psi_{\pi(w)}^{-1}).\qedhere
\]
\end{proof}
\begin{corollary} For~$(w_1,g_1),(w_2,g_2)\in W(\Q)^+\times G(\A_f)$ such that~$[\pi(w_1),g_1]=\sigma([\pi(w_2),g_2])$ we have
\begin{equation}\label{eq:gal invariance}
\delta_{g_1^{-1} T_{\pi(w_1)}g}=\delta_{g_2^{-1} T_{\pi(w_2)}g_2}
\text{ and }
H_f(g_1^{-1}\psi_{\pi(w_1)}^{-1})=H_f(g_2^{-1}\psi_{\pi(w_2)}^{-1}).
\end{equation}
\end{corollary}
\begin{proof} Let~$m\in M_{\pi(w_2)}(\A_f)$ be such that~$m\cdot Z(M_{\pi(w_2)})(\Q)= \rho_{\pi(w_2)}(\sigma)$ in~$M_{\pi(w_2)}(\A_f)/Z(M_{\pi(w_2)})(\Q)$.
We recall that~$\phi_w(m)\in U$.
We have
\[
\sigma([\pi(w_2),g_2])=[\pi(w_2),m\cdot g_2].
\]
There exist~$(\gamma,k)\in G(\Q)\times K$ such that
\[
(\gamma\cdot \pi(w_2),\gamma\cdot m\cdot g_2\cdot k)=(\pi(w_1),g_1).
\]
By Proposition~\ref{prop:K U M invariance}, we deduce~\eqref{eq:gal invariance}.
\end{proof}
\section{Proof of the hybrid conjecture for Shimura varieties abelian type}\label{sec:proof}

The structure of the proof of Th.~\ref{thm:main thm} is essentially the same as in~\cite[\S7]{APZ1} and~\cite[\S3]{APZ2}.

Firstly, we use hybrid orbits instead of generalised Hecke orbits, and geometric hybrid orbits instead of geometric Hecke orbits.

\subsection{} In the steps ``reduction to the Hodge generic case''~\cite[7.1.1]{APZ1} and ``reduction to the adjoint datum''~\cite[7.1.2]{APZ1}  
we use the functoriality properties of hybrid orbits (\cite[Cor. 7.2]{RYLMS}). We have notably: if~$(G',X')\leq (G,X)$ is a subdatum 
then~$\Sigma_X(M,X_M,x)\cap X'=\Sigma_{X'}(M,X_M,x)$ and $\Sigma_{X'^{ad}}(M,X_M,x)=ad_{G'}(\Sigma_X(M,X_M,x))$.
Moreover~$(G',X')$ and~$(G'^{ad},X^{ad})$ are of abelian type if~$(G,X)$ is of abelian type. We may thus pass to a subdatum and assume that~$V$ is Hodge generic in~$(G, X)$ 
and we can pass to the adjoint datum and assume~$G$ is of adjoint type.

\subsection{} 

In~\cite[7.1.3]{APZ2}, ``Induction argument for factorable subvarieties'', we note that~$(G_1\times G_2,X_1\times X_2)$ is of abelian type if and only if~$(G_1,X_1)$ and~$(G_2,X_2)$ are.

\subsection{} In~\cite[\S7.2]{APZ2} we use Th.~\ref{thm:finiteness geometric hybrid} instead of~\cite[Th.~2.4]{APZ2}.

\subsection{} In~\cite[\S\S7.2.1--7.2.4]{APZ2} we use~$W=G/(Z_G(M^{der})\cap N_G(Z(M)))$ associated with a geometric hybrid orbit as in~\S\ref{sec:geometric hybrid orbit}, instead of~$W=G/Z_G(M)$
as in \cite{APZ2}.

\subsection{}

Another change from~\cite[\S\S 3.1--3.3, 7]{APZ1} is in~\cite[7.2.3]{APZ1} where we use our Th.~\ref{thm:tech}, instead of the lower bound on the size of Galois orbits~\cite[Th.~6.4]{APZ1}.

We may apply Th.~\ref{thm:tech} to the Galois image~$U_{x'}$, because~$(G,X)$ is of abelian type.
\subsection{} In the line preceeding of \cite[\S7 (62)]{APZ1} we replace \cite[\S7 Th.~5.16]{APZ2} by Th.~\ref{thm:111}.

\appendix

\end{document}